\documentclass[a4paper,11pt]{amsart}
\usepackage{amsfonts,amssymb,amsmath,amsthm,abstract}
\usepackage[ps,all,arc,rotate]{xy}
\usepackage[lmargin=1in,rmargin=1in,tmargin=1in,bmargin=1in]{geometry}
\usepackage{fancyhdr}
\usepackage{pb-diagram}

\newtheorem{prop}{Proposition}[section]
\newtheorem{lemma}[prop]{Lemma}
\newtheorem{thm}[prop]{Theorem}
\newtheorem{cor}[prop]{Corollary}

\theoremstyle{definition}
\newtheorem{defn}[prop]{Definition}

\newtheorem{rmk}[prop]{Remark}
\newtheorem{ex}[prop]{Example}
\newtheorem{ass}[prop]{Assumption}

\mathchardef\mhyphen="2D
\DeclareMathOperator{\rk}{rk}

\DeclareMathOperator{\spec}{Spec}

\DeclareMathOperator{\proj}{Proj}

\newcommand{\ra}{\rightarrow}

\def\AA{\mathbb A}\def\CC{\mathbb C}
\def\GG{\mathbb G}

\def\NN{\mathbb N}\def\PP{\mathbb P}
\def\RR{\mathbb R}

\def\ZZ{\mathbb Z}

 \def\U{\mathrm{U}}
\def\GL{\mathrm{GL}}

\title{Stratifications associated to reductive group actions on affine spaces}

\author{Victoria Hoskins}

\begin{document}

\maketitle

\begin{abstract}
For a complex reductive group $G$ acting linearly on a complex affine space $V$ with respect to a character $\rho$, we show two stratifications of $V$ associated to this action (and a choice of invariant inner product on the Lie algebra of the maximal compact subgroup of $G$) coincide. The first is Hesselink's stratification by adapted 1-parameter subgroups and the second is the Morse theoretic stratification associated to the norm square of the moment map. We also give a proof of a version of the Kempf--Ness theorem which states that the GIT quotient is homeomorphic to the symplectic reduction (both taken with respect to $\rho$). Finally, for the space of representations of a quiver of fixed dimension, we show that the Morse theoretic stratification and Hesselink's stratification coincide with the stratification by Harder--Narasimhan types.
\end{abstract}

\section{Introduction}

When a complex reductive group $G$ acts linearly on a complex projective variety $X\subset \PP^n$, then Mumford's geometric invariant theory (GIT) \cite{mumford} associates to this action a projective GIT quotient $X /\!/ G$ whose homogeneous coordinate ring is the $G$-invariant part of the homogeneous coordinate ring of $X$. The inclusion of the $G$-invariant subring induces a rational map $X \dashrightarrow X/\!/G$ which restricts to a morphism on the open subset $X^{\text{ss}} \subset X$ of semistable points. Topologically, the projective GIT quotient $X/\!/G$ is $X^{\text{ss}}/G$ modulo the equivalence relation that two orbits are equivalent if and only if their closures meet in $X^{\text{ss}}$. If $X$ is smooth, then it is a symplectic manifold with symplectic form given by restricting the K\"{a}hler form on $\PP^n$ and we can assume without loss of generality that the action of the maximal compact subgroup $K \subset G$ preserves this symplectic form. Then there is an associated moment map $\mu : X \ra \mathfrak{K}^*$ where $\mathfrak{K}$ is the Lie algebra of $K$. The symplectic reduction of Marsden and Weinstein \cite{mw} and Meyer \cite{meyer} is the quotient $\mu^{-1}(0)/K$. The Kempf--Ness theorem (see \cite{kempf_ness} and \cite{mumford} $\S$8) gives a homeomorphism between the symplectic reduction and GIT quotient
\[ \mu^{-1}(0)/K \simeq X^{\text{ss}}/\!/G.\]

Kirwan \cite{kirwan} and Ness \cite{ness} show that two stratifications of the smooth projective variety $X$ associated to this action (and a choice of $K$-invariant inner product on $\mathfrak{K}$) coincide. The first stratification is Hesselink's stratification by adapted 1-parameter subgroups (1-PSs) of $G$. We recall that the Hilbert--Mumford criterion is a numerical criterion used to determine semistability of points in terms of 1-PSs (see \cite{mumford} $\S$2). Kempf builds on these ideas and associates to any unstable point a conjugacy class for a parabolic subgroup of \lq adapted' 1-PSs which are most responsible for the instability of that point \cite{kempf} (where the norm associated to the given inner product is used to give a precise definition of most responsible). Hesselink shows the unstable locus $X-X^{\text{ss}}$ can be stratified by conjugacy classes of 1-PSs  \cite{hesselink}. We view this as a stratification of $X$ with the open stratum given by $X^{\text{ss}}$ (which we can think of as the stratum corresponding to the trivial 1-PS). The second stratification is a Morse theoretic stratification associated to the norm square of the moment map $|| \mu ||^2 : X \ra \RR$ where we use the norm associated to the given inner product. The strata are indexed by adjoint orbits $K \cdot \beta$ for $\beta \in \mathfrak{K}$ (or equivalently coadjoint orbits as the inner product allows us to identify $\mathfrak{K}^* \cong \mathfrak{K}$). For an adjoint orbit $K \cdot \beta$, we let $C_{K \cdot \beta}$ denote the set of subsets of critical points for $|| \mu ||^2$ on which $\mu$ takes a value in the coadjoint orbit corresponding to $K \cdot \beta$. The corresponding stratum $S_{K \cdot \beta}$ consists of all points whose negative gradient flow under $||\mu||^2$ converges to $C_{K \cdot \beta}$. For both stratifications, one can determine the index set for the stratification from the weights of the action of a maximal torus.

In this paper, we suppose $G$ is a complex reductive group acting linearly on a complex affine space $V$ and ask whether the same results still hold. One immediate difference from the above set up is that $V$ is not compact and so we may have some issues with convergence properties; luckily this turns out to not be a problem as in the algebraic setting we are only interested in actions of 1-PSs $\lambda(t)$ on $v \in V$ for which the limit as $t \to 0$ exist and in the symplectic setting the convergence of the negative gradient flow of the norm square of the moment map has already been shown by Harada and Wilkin \cite{haradawilkin} and Sjamaar \cite{sjamaar}.

The affine GIT quotient for the action of $G$ on $V$ is the morphism $V \ra V/\!/ G := \spec \CC[V]^G$ associated to the inclusion of the invariant subring $\CC[V]^G \hookrightarrow \CC[V]$. There is no notion of semistability yet and so the GIT quotient is topologically $V/G$ modulo the equivalence relation that two orbits are equivalent if and only if their closures meet. Therefore Hesselink's stratification of $V$ is the trivial stratification. If we take a Hermitian inner product on $V$ such that the maximal compact subgroup $K \subset G$ acts unitarily, then there is a natural moment map $\mu : V \ra \mathfrak{K}^*$ for the action of $K$. In this case there is only one index $0$ for the Morse stratification associated to $||\mu||^2$ (for any choice of inner product) with critical subset corresponding to $C_0=\mu^{-1}(0)$ and so the Morse stratification is also trivial. Therefore, trivially the stratifications agree. The affine GIT quotient is homeomorphic to the symplectic reduction (this is a special case of Theorem \ref{aff KN} below). However, in the compact setting the proof is given by exhibiting a continuous bijection between the compact symplectic reduction and the separated GIT quotient which is then a homeomorphism. In the affine case, we no longer have compactness and so we instead provide a continuous inverse going in the opposite direction by using the retraction $V \ra \mu^{-1}(0)$ associated to the negative gradient flow for $||\mu||^2$.

More generally, we can use a character $\rho$ of $G$ to get a non-trivial notion of semistability by using $\rho$ to lift the action of $G$ on $V$ to the trivial line bundle $L = V \times \CC$. In this situation there is a GIT quotient $V/\!/_\rho G$ (with respect to $\rho$) which is a quotient of an open subset $V^{\rho \text{-ss}}\subset V$ of $\rho$-semistable points. This construction was used by King to construct moduli space of quiver representations \cite{king}. More precisely, this GIT quotient is the quasi-projective variety
\[V/\!/_\rho G = \proj \bigoplus_{n \geq 0} \CC[V]^G_{\rho^n}\] 
where $\CC[V]^G_{\rho^n}=\{f(g \cdot v) =\rho^n(g)f(v) \text{ for all } v , g \}$ is the ring of semi-invariants of weight $\rho^n$. Thus, a point $v$ is $\rho$-semistable if and only if there is a semi-invariant $f$ of weight $\rho^n$ for $n > 0$ such that $f(v) \neq 0$. For infinite groups the origin is always $\rho$-unstable (for nontrivial characters $\rho$) and so Hesselink's stratification is nontrivial. On the symplectic side, if $\rho$ is a character of $K$, then we can use this to shift the moment map. We write $\mu^{\rho}$ for the shifted moment map and the symplectic reduction (with respect to $\rho$) is $(\mu^{\rho})^{-1}(0)/K$. We give a proof of an affine version of the Kempf--Ness theorem which states that the GIT quotient is homeomorphic to the symplectic reduction (both taken with respect to $\rho$).

Next we compare Hesselink's stratification with the Morse theoretic stratification of $V$ associated to a fixed $K$-invariant inner product on $\mathfrak{K}$. The main difference for Hesselink's stratification in the affine case is that for an unstable point $v$, to determine which 1-PSs are adapted to $v$ we only consider those for which $\lim_{t \to 0} \lambda(t) \cdot v$ exists. On the Morse theory side, since the negative gradient flow of the norm square of the moment map converges much of the picture remains the same as in the projective setting. In the projective setting, Kirwan \cite{kirwan} $\S$6 gave a further description of the Morse strata $S_{K \cdot \beta}$ in terms of Morse strata for the functions $\mu_\beta : V \ra \RR$ given by $v \mapsto \mu(v) \cdot \beta$ and $||\mu-\beta||^2$. In the affine case, we also provide a similar description; however, our description differs slightly to that of Kirwan due to the fact that the negative gradient flow of $\mu_\beta$ on $V$ does not always converge. We prove the Morse theoretic stratification coincides with Hesselink's stratification (both taken with respect to $\rho$) for a fixed $K$-invariant inner product on $\mathfrak{K}$. Furthermore, we show that the index set for Hesselink's stratification and the Morse theoretic stratification can be determined combinatorially from the weights associated to the action of a maximal torus (cf. $\S$\ref{hess ind} and $\S$\ref{sec desc indices} below).

Finally, we apply this to the case in which $V$ is the space of representations of a quiver of fixed dimension and $G$ is a reductive group acting on $V$ such that the orbits correspond to isomorphism classes of representations. It follows from above that the Morse stratification coincides with Hesselink's stratification, but we can also compare this to the stratification by Harder--Narasimhan types (where the notion of Harder--Narasimhan filtration depends on a choice of invariant inner product on the Lie algebra of the maximal compact subgroup of $G$). For a fixed inner product, we prove that all three stratifications coincide. For the case when the chosen inner product is the Killing form, the Harder--Narasimhan stratification has been described by Reineke in \cite{reineke} and Harada and Wilkin show that for this inner product the Harder--Narasimhan stratification and Morse theoretic stratification on $V$ coincide \cite{haradawilkin}. In his Ph.D. thesis, Tur \cite{tur} shows that Hesselink's stratification by adapted 1-PSs agrees with the Harder--Narasimhan stratification (for any choice of invariant inner product) and so it follows from this and the result we gave above that all three stratifications coincide. However we provide a concise proof of this fact for completeness of the paper. Whilst this paper was being completed, we note that Zamora has also given a proof that the Kempf filtration (this is a natural filtration associated to an adapted 1-PS) is equal to the Harder--Narasimhan filtration for quiver representations \cite{zamora}.

The layout of this paper is as follows. In $\S$\ref{sec GIT}, we give some results on affine geometric invariant theory. In particular, we give results for semistability of points with respect to a character $\rho$ and also describe Hesselink's stratification by adapted 1-PSs. In $\S$\ref{sec sym}, we describe the moment map, symplectic reduction and the Morse stratification associated to the norm square of the moment map for complex affine spaces. In $\S$\ref{sec corr}, we show for a fixed invariant inner product that the Morse stratification agrees with Hesselink's stratification. In this section we also give a proof of the affine Kempf---Ness theorem and prove an alternative description of the Morse strata. In $\S$\ref{sec quiv}, we apply the above to the space of representations of a quiver of fixed dimension. We define a notion of Harder--Narasimhan filtration which depends on the choice of invariant inner product and show that the stratification by Harder--Narasimhan types coincides with both the Morse stratification and Hesselink's stratification.

\subsection*{Acknowledgements}
I would like to thank Frances Kirwan for teaching me about these results in the projective setting. I also wish to thank Ben Davison, Alastair King and Graeme Wilkin for useful discussions and comments.  

\section{Affine geometric invariant theory}\label{sec GIT}

In this section we can work over an arbitrary algebraically closed field $k$ of characteristic zero. Let $G$ be a reductive group acting linearly on an affine space $V$ over $k$. Let $k[V]$ denote the $k$-algebra of regular functions on $V$; then there is an induced action of $G$ on $k[V]$ given by $g \cdot f(v) = f( g^{-1} \cdot v) $ for $f \in k[V]$ and $g \in G$. The inclusion of the invariant subalgebra $k[V]^G \hookrightarrow k[V]$ induces a morphism of affine varieties $V \ra V/\!/G$ which is known as the affine GIT quotient. In general this is not the same as the topological quotient $V/G$ as the affine GIT quotient identifies orbits whose closures meet. In particular if one orbit is contained in the closure of every other orbit (as is the case when $\GG_m$ acts on $\AA^n$ by scalar multiplication), then the affine GIT quotient $V/\!/G$ is simply a point.

To avoid such collapsing for reductive actions on affine spaces, we can instead use a non-trivial character $\rho : G \ra \GG_m$ to linearise the action so that we obtain a better quotient of an open subset of $V$ as follows. Let $L = V \times k$ denote the trivial line bundle on $L$; then we use $\rho$ to lift the action of $G$ on $V$ to $L$ so that
\[ g \cdot (v,c) = (g \cdot v, \rho(g)c)\]
for $g \in G$ and $(v,c) \in L = V \times k$. We write $L_\rho$ to denote the linearisation consisting of the line bundle $L$ and lift of the $G$-action given by $\rho$. We note that as linearisations $ L_\rho^{\otimes n}=L_{\rho^n}$ for all $n \in \ZZ$ where $L^{\otimes -1}_\rho:=L^{-1}_\rho=L_{\rho^{-1}}$ is the dual linearisation to $L_\rho$. There is an induced action of $G$ on $H^0(V, L_\rho^{\otimes n})$ given by
\[ g \cdot \sigma(v) = \rho^n(g) \sigma(g^{-1} v) \]
for $\sigma \in H^0(V, L_\rho^{\otimes n})$ and $g \in G$. We note that the invariant sections 
\[ H^0(V, L_\rho^{\otimes n})^G \cong k[V]^G_{\rho^n} := \{f \in k[V] : f(g \cdot v) = \rho^n(g)f(v) \text{ for all } v \in V, g \in G \} \]
are equal to the semi-invariants on $V$ of weight $\rho^n$.
Consider the graded algebra \[R := \oplus_{n \geq 0} H^0(V,L_\rho^{\otimes n})\] and its invariant graded subalgebra $R^G = \oplus_n R^G_n$ where $R^G_n=k[V]^G_{\rho^n}$ is the algebra of semi-invariants on $V$ of weight $\rho^n$. The inclusion $R^G \hookrightarrow R$ induces a rational map 
\begin{equation}\label{rat map}
V \dashrightarrow V/\!/_\rho G = \proj R^G \end{equation} 
which is undefined on the null cone
\begin{equation}\label{null cone} N = \{ v \in V : f(v) = 0 \:  \forall \: f \in \oplus_{n>0} R^G_n \}. \end{equation}
Following Mumford \cite{mumford}, Definition 1.7 we have:

\begin{defn} Let $v \in V$; then
\begin{enumerate}
\renewcommand{\labelenumi}{\roman{enumi})}
\item $v$ is $\rho$-{semistable} if there is an invariant section $\sigma \in H^0(V, {L}_\rho^{\otimes n})^G = k[V]^G_{\rho^n}$ for some $n > 0$ such that $\sigma(v) \neq 0$.
\item  $v $ is $\rho$-{stable} if $\dim G_x=0$ and there is an invariant section $\sigma \in H^0(V,{L}_\rho^{\otimes n})^G= k[V]^G_{\rho^n}$ for some $n > 0$ such that $\sigma(v) \neq 0$ and the action of $G$ on the open affine subset $V_\sigma:= \{ u \in V : \sigma(u) \neq 0 \}$ is closed (that is, all $G$-orbits in $V_\sigma$ are closed).
\item The points which are not $\rho$-semistable are called $\rho$-{unstable}.
\end{enumerate}
The open subsets of $\rho$-stable and $\rho$-semistable will be denoted by $V^{\rho \text{-s}}$ and $V^{\rho\text{-ss}}$ respectively.
\end{defn}

By definition, the semistable locus $V^{\rho \text{-ss}}$ is the complement of the null cone $N$ and we call the morphism $V^{\rho \text{-ss}} \ra V/\!/_\rho G$ the GIT quotient with respect to $\rho$. This approach of using a character to twist the trivial linearisation on an affine space was used by King to construct moduli spaces of representations of finite dimensional algebras \cite{king}. We note that King modifies the notion of stability so that one can have a subgroup of dimension greater than zero contained in the stabiliser of each point and still have stable points. In fact Mumford's original notion of stability \cite{mumford} does not require points to have zero dimensional stabilisers; the modern notion of stability (where one asks for zero dimensional stabilisers) is what Mumford refers to as proper stability.

\begin{thm}(Mumford)
The GIT quotient $\varphi : V^{\rho\text{-}\mathrm{ss}} \ra V/\!/_\rho G$ is a good quotient for the action of $G$ on $V^{\rho\text{-}\mathrm{ss}}$. Moreover, there is an open subset $V^{\rho\text{-}\mathrm{s}}/G \subset V/\!/_\rho G$ whose preimage under $\varphi$ is $V^{\rho\text{-}\mathrm{s}}$ and the restriction $\varphi : V^{\rho\text{-}\mathrm{s}} \ra V^{\rho\text{-}\mathrm{s}}/G$ is a geometric quotient (which in particular is an orbit space).
\end{thm}

\begin{rmk} In general the GIT quotient $V/\!/_\rho G= \proj R^G$ with respect to $\rho$ is only quasi-projective. It is projective over the affine GIT quotient $\spec k[V]^G = \spec R^G_{0}$ and so is projective if $k[V]^G = k$. We note that if $\rho$ is the trivial character, then in this construction we recover the affine GIT quotient $V \ra V/\!/ G$.
\end{rmk}

\subsection{Criteria for stability}

We note that for finite groups, the notion of semistability is trivial for any character $\rho$ and so from now on we may as well assume our group is infinite. The following lemma gives a topological criterion for semistability. The proof follows in the same way as the original projective version (see \cite{mumford} Proposition 2.2 and also \cite{king} Lemma 2.2). We let $L_{\rho}^{-1} = L_{\rho^{-1}}$ denote the dual linearisation to $L_{\rho}$. 

\begin{lemma}\label{top crit}
Let $\tilde{v} = (v,a) \in L_{\rho}^{-1}$ be a point lying over $v \in V$ for which $a \neq 0$. Then
\begin{enumerate}
\renewcommand{\labelenumi}{\roman{enumi})}
\item $v$ is $\rho$-semistable if and only if the orbit closure $\overline{G \cdot \tilde{v}}$ of $\tilde{v}$ in $L_{\rho}^{-1}$ is disjoint from the zero section $V \times \{0\} \subset L_{\rho}^{-1} = V \times k$.
\item $v$ is $\rho$-stable if and only if the orbit $G \cdot \tilde{v}$ of $\tilde{v}$ in $L_{\rho}^{-1}$ is closed and $ \dim G \cdot \tilde{v}  = \dim G$.
\end{enumerate}
\end{lemma}

The topological criterion can be reformulated as a numerical criterion, known as the Hilbert--Mumford criterion, by using 1-parameter subgroups (1-PSs) of $G$; that is, non-trivial homomorphisms $\lambda : \GG_m \ra G$. As every semi-invariant for the action of $G$ is also a semi-invariant for the action of any subgroup of $G$, we see that if $v$ is $\rho$-(semi)stable for the action of $G$ it must also be $\rho$-(semi)stable for the action of $\lambda(\GG_m)$ for any 1-PS $\lambda$. We shall see that the notion of $\rho$-(semi)stability for a one-dimensional torus $\lambda(\GG_m)$ can easily be reformulated as a numerical condition. By the topological criterion for the one-dimensional torus $\lambda(\GG_m)$:
\begin{enumerate}
\renewcommand{\labelenumi}{\roman{enumi})}
\item $v$ is $\rho$-(semi)stable for the action of $\lambda(\GG_m)$ if and only if $\lim_{t \to 0} \lambda(t) \cdot \tilde{v} \notin V \times \{ 0 \}$ and $\lim_{t \to \infty} \lambda(t) \cdot \tilde{v} \notin V \times \{ 0 \}$.
\item $v$ is $\rho$-stable for the action of $\lambda(\GG_m)$ if and only if neither limit exits. 
\end{enumerate}

Let $(\rho, \lambda)$ denote the integer such that $\rho \circ \lambda(t) = t ^{(\rho, \lambda)}$; then
\[ \lambda(t) \cdot \tilde{v} = \left(\lambda(t) \cdot v, t^{-(\rho, \lambda)} a \right) \in L^{-1}_\rho. \]
The limit $\lim_{t \to 0} \lambda(t) \cdot \tilde{v}$ exists if and only if $\lim_{t \to 0} \lambda(t) \cdot v$ exists and $(\rho, \lambda) \leq 0$. In particular if $\lim_{t \to 0} \lambda(t) \cdot v$ exists, then $\lim_{t \to 0} \lambda(t) \cdot \tilde{v} \notin V \times \{ 0 \}$ if and only if $(\rho,\lambda) \geq 0$ (if it is positive the limit does not exist and if it is zero then the limit exists but does not belong to the zero section).

For $v \in V$ and a 1-PS $\lambda$ of $G$, we let $\mu^{\rho}(v,\lambda) :=(\rho, \lambda)$. We note that the quantity $\mu^{\rho}(v,\lambda)$ is independent of the point $v$, but we use this notation in analogy with the projective case.

\begin{prop}\label{HM crit}(Hilbert--Mumford criterion)
For $v \in V$ and a 1-PS $\lambda$ of $G$, we have
\begin{enumerate}
\renewcommand{\labelenumi}{\roman{enumi})}
\item v is $\rho$-semistable if and only if $\mu^{\rho}(v, \lambda)  \geq 0$  for every 1-PS $\lambda$ of $G$ for which the limit $\lim_{t \to 0} \lambda(t) \cdot v $ exists.
\item v is $\rho$-semistable if and only if $\mu^{\rho}(v, \lambda)  > 0$  for every 1-PS $\lambda$ of $G$ for which the limit $\lim_{t \to 0} \lambda(t) \cdot v $ exists.
\end{enumerate}
\end{prop}

The discussion proceeding this proposition proves the \lq only if' direction of the Hilbert--Mumford criterion and the converse follows by using the following well-known theorem from  GIT (the first version of this goes back to Hilbert; for a proof see \cite{kempf}, Theorem 1.4 or \cite{mumford} $\S$2):

\begin{thm}\label{fund thm}
Let $G$ be a reductive group acting linearly on an affine space $V$. If $v \in V$ and $w \in \overline{G \cdot v}$, then there is a 1-PS $\lambda$ of $G$ such that $\lim_{t \to 0} \lambda(t) \cdot v$ exists and is equal to $w$.
\end{thm}

If we fix a maximal torus $T \subset G$, then the action of $T$ on $V$ gives a weight decomposition
\begin{equation}\label{wt decomp}
V = \bigoplus_{\chi \in \chi^*(T)} V_\chi 
\end{equation} 
where $ \chi^*(T):=\text{Hom}(T,\GG_m)$ and $V_\chi := \{v \in V : t \cdot v = \chi(t)v \: \forall \: t \in T \}$. We refer to the finite set of characters $\chi$ for which $V_\chi \neq 0$ as the $T$-weights for the action. For any $v \in V$, we write $v = \sum v_\chi$ with respect to this decomposition and define $\text{wt}_T(v) = \{ \chi : v_\chi \neq 0 \}$. We let $\chi_*(T) := \text{Hom}(\GG_m,T)$ denote the set of cocharacters. If $T \cong (\GG_m)^n$ is an $n$-dimensional torus, then there are natural identifications \[\chi_*(T) \cong \ZZ^n \cong \chi^*(T)\]
\[ \lambda(t) = (t^{m_1}, \dots ,t^{m_n}) \mapsto (m_1, \dots , m_n) \mapsto \chi(t_1, \dots, t_n) = \Pi t_i^{m_i} \] 
and the natural pairing between characters and cocharacters corresponds to the dot product on $\ZZ^n$. We define the cone of allowable 1-PSs for $v$ to be
\begin{equation}\label{cone allow} C_v := \bigcap_{\chi \in \text{wt}_T(v)} H_\chi \subset \RR^n\end{equation}
where $ H_\chi := \{ \lambda \in \chi_*(T)_\RR\cong \RR^n : (\chi, \lambda) \geq 0 \}$. Then, by construction of this cone, a 1-PS $\lambda$ belongs to $C_v$ if and only if \[\lim_{t \to 0} \lambda(t) \cdot v = \sum_\chi \lim_{t \to 0} t^{(\chi,\lambda)} v_\chi\]
exists. Then the Hilbert--Mumford criterion can be restated as:
\begin{enumerate}
\renewcommand{\labelenumi}{\roman{enumi})}
\item  $v$ is $\rho$-semistable for the action of $T$ if and only if $C_v \subset H_\rho$. 
\item $v$ is $\rho$-stable for the action of $T$ if and only if $C_v - \{ 0 \} \subset H_\rho^o$ where \[ H_\rho^o:= \{ \lambda \in  \chi_*(T)_\RR\cong \RR^n  :  ( \rho, \lambda) >0 \}.\]
\end{enumerate}
Moreover, as every 1-PS of $G$ is conjugate to a 1-PS of $T$, we can use the above criteria to give criteria for $\rho$-(semi)stability with respect to $G$.

\begin{prop}
Let $v \in V$; then
\begin{enumerate}
\renewcommand{\labelenumi}{\roman{enumi})}
\item $v$ is $\rho$-semistable if and only if for all $g \in G$ we have $C_{gv} \subset H_\rho$.
\item $v$ is $\rho$-stable if and only if for all $g \in G$ we have $C_{gv} - \{ 0 \} \subset H_\rho^o$.
\end{enumerate}
\end{prop}

\begin{rmk} If $C_{gv} = \{0\}$ for all $g \in G$, then $G \cdot v$ is $\rho$-stable for any character $\rho$. We shall refer to such points as strongly stable points. We note that if $v = 0$, then $C_ 0 := \RR^n$ and so $0$ is unstable for any non-trivial character $\rho$. We also note that if $\rho = 0$ is the trivial character, then $H_\rho = \RR^n$ and $H^o_\rho = \emptyset$, thus all points are semistable and the strongly stable points are the only stable points.
\end{rmk}

In the projective setting, the set of torus weights can also be used to study (semi)stability of points (for example, see \cite{dolgachev} Theorem 9.3); although a weight polytope is used (which is equal to the convex hull of some torus weights) rather than a cone of allowable 1-PSs.

\subsection{Instability}

To study the null cone $N$ of unstable points, Kempf gave a notion for a 1-PS to be adapted to an unstable point which makes use of a normalised Hilbert--Mumford function $\mu^{\rho}(v,\lambda)/|| \lambda ||$ where $|| - ||$ is a fixed norm on the set $\chi_*(G)/G$ of conjugacy classes of 1-PSs. As the set of orbits for the action of $G$ on $G$ (by conjugation) is equal to the set of orbits for the action of the Weyl group $W$ on a maximal torus $T$ of $G$, to fix such a norm is equivalent to fixing a maximal torus $T$ of $G$ and a norm on $\chi_*(T)/W$. Under the natural identification $\chi_*(T) \cong \ZZ^n$, we can take the standard Euclidean norm on $\RR^n$ and average over the action of the (finite) Weyl group $W$ to produce a norm which is invariant under the action of $W$. 

\begin{ass}\label{int ass} We assume that $||\lambda||^2 \in \ZZ$ for all 1-PSs $\lambda$.
 \end{ass}

\begin{ex}
If $G = \GL(n)$ and $T$ is the maximal torus of diagonal matrices, then the dot product on $\ZZ^n \cong \chi_*(T)$ is invariant under the Weyl group $W = S_n$ which acts on $T$ by permuting the diagonal entries.
\end{ex}

\begin{rmk}
Often we shall work over the complex numbers, in which case a complex reductive group $G$ is equal to the complexification of its maximal compact subgroup $K \subset G$. As any 1-PS of $G$ is conjugate to a 1-PS which sends the maximal compact subgroup $\U(1) \subset \CC^*$ to $K \subset G$, we can just consider 1-PSs of this form. There is a natural isomorphism $\chi_*(K)_{\RR} = \text{Hom}(U(1),K) \otimes_\ZZ {\RR} \cong \mathfrak{K}:= \mathrm{Lie} K$ given by
\[ \lambda \mapsto d\lambda(2 \pi i) :=\frac{d}{dt} \lambda(\exp(2\pi i t))|_{t=0} \in \mathfrak{K} \]
whose inverse is given by sending $\alpha \in \mathfrak{K}$ to the real 1-PS $\exp(\RR\alpha) \subset K$ (although we note that the map $\lambda : S^1 \ra K$ given by $\lambda(\exp(2 \pi it)) = \exp( t \alpha)$ is not necessarily a group homomorphism). We refer to points $\alpha \in \mathfrak{K}$ which lie in the image of $\chi_*(K)$ under this isomorphism as integral weights (in this case, the corresponding $\lambda$ is a group homomorphism). A $G$-invariant norm on the set of 1-PSs of $G$ is then equivalent to a $K$-invariant norm on $\mathfrak{K}$ and Assumption \ref{int ass} is equivalent to requiring that $||\alpha||^2 \in \ZZ$ for all integral weights $\alpha$. For example, any positive scalar multiple of the Killing form on $\mathfrak{K}$ is a $K$-invariant inner product and a suitably scaled version of the Killing form would satisfy Assumption \ref{int ass}.
\end{rmk}

\begin{defn}\label{def adapted}
Let $v$ be $\rho$-unstable for the action of $G$ and define
\[ M^\rho_G(v) = \inf \frac{\mu^{\rho}(v,\lambda)}{|| \lambda ||} \]
where the infimum is taken over all 1-PSs $\lambda$ of $G$ for which $\lim_{t \to 0} \lambda(t) \cdot v$ exist. We say that a subgroup $H \subset G$ is optimal for $v$, if $M^{\rho}_G(v) = M^{\rho}_H(v)$. We shall often write $M^\rho(v)$ rather than $M^\rho_G(v)$ when it is clear that we are considering the action of a fixed group $G$. A 1-PS $\lambda$ is said to be $\rho$-adapted to $v$ if it achieves this minimum value; that is,
\[ M^\rho_G(v) = \frac{\mu^{\rho}(v,\lambda)}{|| \lambda ||} .\]
We let $\wedge^\rho(v) $ denote the set of indivisible 1-PSs which are $\rho$-adapted to $x$.
\end{defn}

For torus actions, we can combinatorially determine $\wedge^\rho(v) $ from the cone $C_v$ defined at (\ref{cone allow}):

\begin{lemma}\label{determine wts Hess}
Let $v$ be an unstable point for the action of a torus $T \cong \GG_m^n$ on a vector space $V$ linearised by a non-trivial character $\rho$; then $\wedge^\rho(v) $ consists of a unique indivisible 1-PS.
\end{lemma}
\begin{proof}
As $v$ is $\rho$-unstable, the cone of allowable 1-PSs $C_v $ defined at (\ref{cone allow}) is not contained in the half space $H_\rho$. By definition
\[ M^\rho(v) = \inf \frac{(\rho,\lambda)}{||\lambda||} \]
where the infimum is taken over 1-PSs $\lambda \in C_v \cap \ZZ^n - \{ 0 \}$. As $(\rho,\lambda) = ||\lambda || \: || \rho || \cos \theta_{\lambda,\rho}$ where $\theta_{\lambda,\rho}$ is the angle between the vectors $\lambda$ and $\rho$, this is equivalent to taking the supremum over $\lambda \in C_v \cap \ZZ^n - \{ 0 \}$ of the angle between $\rho$ and $\lambda$. Hence, there is a unique ray $R$ contained in the cone $C_v$ along which the normalised Hilbert--Mumford function is minimised. Since $R$ is the ray consisting of vectors whose angle to $\rho$ is maximal over all vectors contained in the cone $C_v$, either $R$ is the ray spanned by $-\rho$ or $R$ is a ray contained in the boundary of the cone (that is; it is contained in $\chi^{\perp} := \{ \alpha \in \RR^n : (\alpha,\rho) = 0 \}$ for some $\chi \in \text{wt}_T(v)$). In either case, $R$ is the span of an integral point in $\RR^n \cong \chi_*(T)_\RR$ and so contains a unique point which corresponds to a indivisible 1-PS.
\end{proof}

As any two maximal tori of a reductive group $G$ are conjugate and $\mu^\rho(v,\lambda) = \mu^\rho(g \cdot v, g \lambda g^{-1})$, we see that one can calculate the possible values of $M^\rho$ by fixing a maximal torus and calculating \[ \inf_{\lambda \neq 0 \in C} \frac{(p,\lambda)}{||\lambda||} \]
for all cones $C$ which are constructed using some subset of the $T$-weights as at (\ref{cone allow}). In particular, we see that for any reductive group $M^\rho : V \ra \RR$ takes only finitely many values and $\wedge^\rho(v) \neq \emptyset$ for any $\rho$-unstable point.

\begin{defn}
For any 1-PS $\lambda$ of $G$ we define a parabolic subgroup of $G$:
\[ P(\lambda) := \left\{ g \in G : \lim_{t \rightarrow 0} \lambda(t) g \lambda(t^{-1}) \mathrm{\:exists \: in \: } G \right\}. \]
\end{defn}

One can easily modify the results of Kempf \cite{kempf} on the sets $\wedge(v)$ to our setting of a reductive group action $G$ on an affine space $V$ linearised by a character $\rho$:

\begin{thm}\label{kempf thm}(Kempf)
Let $G$ be a reductive group acting on an affine $V$ and suppose $\rho$ is a character which we use to linearise this action. If $v \in V$ is $\rho$-unstable, then
\begin{enumerate}
\renewcommand{\labelenumi}{\roman{enumi})}
\item $\wedge^\rho(g \cdot v) = g \wedge^\rho(v) g^{-1}$ for all $g \in G$.
\item There is a parabolic subgroup $P(v,\rho)$ such that $P(v,\rho) = P(\lambda)$ for all $\lambda \in \wedge^\rho(v)$.
\item All elements of $\wedge^\rho(v)$ are conjugate to each other by elements of $P(v,\rho)$. Moreover, the stabiliser subgroup $G_v$ is contained in $P(v,\rho)$.
\item Let $T \subset P(v,\rho)$ be a maximal torus of $G$, then there is a unique 1-PS of $T$ which belongs to $\wedge^\rho(v)$.
\item If $\lambda \in \wedge^\rho(v)$ and $w = \lim_{t \to 0} \lambda(t) \cdot v$, then also $\lambda \in \wedge^\rho(w)$.
\end{enumerate}
\end{thm}

\subsection{The stratification of Hesselink}\label{sec hess}

When a reductive group $G$ acts linearly on a projective space $\PP^n$, there is a stratification of the unstable locus $\PP^n - (\PP^n)^{ss} = \sqcup S_{[\lambda],d} $ associated to this action and any norm as above which is indexed by pairs $([\lambda],d)$ where $[\lambda]$ denotes the conjugacy class of a 1-PS of $G$ and $d$ is a positive integer (see Hesselink \cite{hesselink} and also Kirwan \cite{kirwan} $\S$12). In this projective case, the strata are defined as
\[ S_{[\lambda],d} := \{ x: M(x) =- d \text{ and } \wedge\!(v) \cap [\lambda] \neq \emptyset  \}\]
where $M(x):= \inf \mu(x,\lambda')/ || \lambda' ||$. We can easily modify these ideas to our setting of a reductive group $G$ acting on an affine space $V$ with respect to a character $\rho$ to obtain a stratification of the null cone $N = V - V^{\rho-ss}$. We fix a norm on the set of conjugacy classes of 1-PSs of $G$ which we assume comes from a conjugation invariant inner product on the set of 1-PSs of $G$. In the affine setting, the index $d$ is redundant as it is determined by $\lambda$ and $\rho$; that is,
\[ d= - M^\rho(v) = - \frac{ \mu^\rho(v,\lambda)}{ ||\lambda||} = - \frac{ (\rho,\lambda)}{||\lambda||} >0 \]
for $\lambda \in \wedge^\rho(v)$. We define
\[ S_{[\lambda]} := \{ v \in N : \wedge^\rho(v) \cap [\lambda] \neq \emptyset \}; \]
then these subsets stratify the unstable locus $N:=V -V^{\rho-ss}$ into $G$-invariant subsets (for details on what we mean by a stratification see Theorem \ref{hess thm} below). The strata $S_{[\lambda]}$ may also be described in terms of the \lq blades'
\begin{equation}\label{Slambda}
  S_{\lambda}= \{ v \in N :  \lambda \in \wedge^\rho(v)\} \subset S_{[\lambda]} 
\end{equation}
and the limit sets
\begin{equation} \label{Zlambda}
  Z_{\lambda} = \{ v \in N :  \lambda \in \wedge^\rho(v) \cap G_v\} \subset S_{\lambda} 
\end{equation}
where $G_v$ denotes the stabiliser subgroup of $v$ and $\lambda$ is a representative of the conjugacy class $[\lambda]$. It follows from Theorem \ref{kempf thm} that $S_{[\lambda]} = GS_\lambda$. We let $p_{\lambda} : S_{\lambda} \ra Z_{\lambda} $ be the retraction given by sending a point $v$ to $\lim_{t \to 0} \lambda(t) \cdot v$. There is a strict partial ordering $<$ on the indices given by $[\lambda] < [\lambda']$ if 
\[ \frac{(\rho,\lambda)}{||\lambda||} > \frac{(\rho,\lambda')}{||\lambda'||}; \]
that is, $M^\rho(v) > M^\rho(v')$ for $v \in S_{[\lambda]}$ and $v' \in S_{[\lambda']}$.

\begin{thm}\label{hess thm}
Given a reductive group $G$ acting on an affine space $V$ with respect to a character $\rho$, there is a decomposition of the null cone $N:=V -V^{\rho\text{-}\mathrm{ss}}$
\[ N = \bigsqcup_{[\lambda]} S_{[\lambda]} \]
into finitely many disjoint $G$-invariant locally closed subvarieties $S_{[\lambda]}$ of $V$. Moreover, the strict partial ordering describes the boundary of a given stratum:
\[ \partial S_{[\lambda]} \cap S_{[\lambda']} \neq \emptyset\]
only if $[\lambda] < [\lambda']$. (We refer to such a decomposition with such an ordering $<$ as a stratification). 
\end{thm}
\begin{proof} By Theorem \ref{kempf thm}, the strata are disjoint and $G$-invariant. Let $\lambda$ be a 1-PS of $G$ which indexes a Hesselink stratum $S_{[\lambda]}$; then we write $V = \oplus_r V_r$ where $V_r=\{ v \in V : \lambda(t) \cdot v = t^r v \}$. We let $V^\lambda = V_0$ denote the fixed point locus for the action of $\lambda$ and $V^{\lambda}_+ = \oplus_{r \geq 0} V_r$ denote the locus consisting of points $v \in V$ for which $\lim_{t \to 0} \lambda(t) \cdot v$ exists. Then both $V^\lambda$ and $V^\lambda_+$ are closed subsets of $V$. There is a projection $p_\lambda : V_+^\lambda \ra V^\lambda $ given by sending a point to its limit under $\lambda(t)$ as $t \to 0$. Clearly we have a containment of the blade $S_\lambda \subset V_+^\lambda$ and its limit set $Z_\lambda \subset V^\lambda$. In fact, $Z_\lambda$ is an open subset of $V^{\lambda}$ and $S_\lambda = p_\lambda^{-1} (Z_\lambda)$ by Proposition \ref{prop on Zlambda} and Lemma \ref{lemma on Slambda} below. It follows from this that $S_{[\lambda]} = GS_\lambda$ is a locally closed subvariety of $V$.

To prove the claim about the boundary, we note that as $S_{[\lambda]}$ is contained in the closed set $G V_+^{\lambda}$, so is its closure. If $v \in \partial S_{[\lambda]} \cap S_{[\lambda']}$; then $v' = g \cdot v \in V_+^\lambda$. As $\lim_{t \to 0} \lambda(t) \cdot v'$ exists
\[ M^\rho(v)=M^\rho(v') = \frac{(\rho,\lambda')}{||\lambda'||} < \frac{(\rho,\lambda)}{||\lambda||};\]
that is, $[\lambda] < [\lambda']$.
\end{proof}

\begin{rmk}
From this stratification, we get a stratification
\begin{equation}\label{hess strat} V = \bigsqcup S_{[\lambda]} \end{equation}
where the minimal stratum $S_{[0]}:= V^{\rho\text{-ss}}$ is open and is indexed by the trivial 1-PS. We shall refer to this stratification of $V$ as Hesselink's stratification.
\end{rmk}

If $\lambda$ is a non-trivial indivisible 1-PS which indexes a Hesselink stratum, we let $G_{\lambda}$ denote the subgroup of $G$ consisting of elements $g \in G$ which commute with $\lambda(t)$ for all $t$; then this group acts on the fixed point locus
\[V^{\lambda}:= \{ v \in V : \lambda(\GG_m) \subset G_v \}\]
and we can use a character $\rho_\lambda$ of $G_\lambda$ to linearise the action. We use $\lambda^*$ to denote the character of $G_\lambda$ which is dual to the 1-PS $\lambda$ of $G_\lambda$ under the given inner product and (using additive notation for the group of characters) define
\begin{equation}\label{defn rholambda}
 \rho_\lambda :=||\lambda||^2 \rho -(\rho,\lambda) \lambda^*.
\end{equation}
We note that $||\lambda||^2$ and $-(\rho,\lambda)$ are both positive integers and so $\rho_{\lambda}$ is a character of $G_\lambda$. We want to compare the semistable set $(V^\lambda)^{\rho_\lambda\text{-ss}}$ for this action of $G_\lambda$ with the limit set $ Z_{\lambda} \subset V^\lambda$ consisting of points $v$ fixed by $\lambda$ and for which $\lambda$ is $\rho$-adapted to $v$. For the projective version of the following proposition, see Kirwan \cite{kirwan} Remark 12.21 and Ness \cite{ness} Theorem 9.4.

\begin{prop}\label{prop on Zlambda}
Let $\lambda$ be a non trivial 1-PS which indexes a Hesselink stratum. Then the limit set $Z_{\lambda}$ is equal to the semistable subset for the action of $G_\lambda$ on $V^{\lambda}$ with respect to the linearisation given by $\rho_\lambda$.
\end{prop}
\begin{proof}
Let $v \in V^{\lambda}$; then firstly we claim that $G_\lambda \subset G$ is optimal for $v$ in the sense of Definition \ref{def adapted}. By Theorem \ref{kempf thm} ii), if $\lambda' \in \wedge^\rho(v)$ then $\lambda'$ is a 1-PS of $P(v,\rho)$ and by iii) of the same theorem, $ G_v \subset P(v,\rho)$. Hence, there exists $p \in P(v,\rho)$ such that $p\lambda'p^{-1}$ and $\lambda$ commute; that is, $p\lambda'p^{-1} \in \chi_*(G_\lambda) \cap \wedge^\rho(v)$ which proves the claim.

Firstly, suppose that $\lambda$ is not $\rho$-adapted to $v \in V^\lambda$ so that $v \notin Z_{\lambda}$. Then there is a 1-PS $\lambda'$ for which $\lim_{t \to 0} \lambda'(t) \cdot v$ exists and such that
\begin{equation}\label{eqn in NK}
 \frac{(\rho,\lambda')}{||\lambda'||} < \frac{(\rho,\lambda)}{||\lambda||} 
\end{equation}
and as $G_\lambda$ is optimal for $v$, we can assume that $\lambda'$ is a 1-PS of $G_\lambda$. By definition of $\rho_\lambda$ we have
\[ \mu^{\rho_\lambda}(v,\lambda') = (\rho_\lambda,\lambda') = ||\lambda||^2(\rho,\lambda') - (\rho,\lambda)(\lambda,\lambda') \]
and using (\ref{eqn in NK}) this gives
\[  \mu^{\rho_\lambda}(v,\lambda')< (\rho,\lambda) \left( ||\lambda|| \: ||\lambda'|| - (\lambda,\lambda') \right). \]
Then since $(\lambda,\lambda') = ||\lambda|| \: || \lambda'|| \cos \theta_{\lambda,\lambda'} \leq ||\lambda|| \:||\lambda'||$ and $(\rho,\lambda) <0$, we have
\[  \mu^{\rho_\lambda}(v,\lambda')< (\rho,\lambda) \left( ||\lambda|| \: ||\lambda'|| - (\lambda,\lambda') \right) \leq 0 \]
which implies $v$ is not $\rho_{\lambda}$-semistable by Proposition \ref{HM crit}.

Conversely, if $v$ is unstable with respect to the action of $G_\lambda$ linearised by $\rho_\lambda$ then there is a 1-PS $\lambda'$ of $G_\lambda$ such that
\[ \mu^{\rho_\lambda}(v,\lambda'):= (\rho_\lambda,\lambda')< 0. \]
Take a maximal torus $T$ of $G_\lambda$ which contains both $\lambda$ and $\lambda'$; then in $\RR^n \cong\chi^*(T)_\RR \cong \chi_*(T)$ consider the vectors $\rho$, $\rho_\lambda$, $\lambda$ and $\lambda'$. By definition of $\rho_\lambda$, this vector is contained in the interior of the cone spanned by $\rho$ and $\lambda$ and is orthogonal to $\lambda$:
\[ (\rho_\lambda,\lambda) = ||\lambda||^2 (\rho,\lambda) -(\rho,\lambda)( \lambda,\lambda) = 0. \]
As $(\rho,\lambda)< 0$ and $(\rho_\lambda,\lambda')< 0$ we also know that the angles $\theta_{\rho,\lambda}$ and $\theta_{\rho_\lambda,\lambda'}$ are both at least $\pi/2$; therefore $\theta_{\rho,\lambda + \lambda'} > \theta_{\rho,\lambda} > \pi/2$ and so
\[ \frac{\mu^\rho(v,\lambda+\lambda')}{||\lambda+\lambda'||} = ||\rho|| \cos(\theta_{\rho,\lambda + \lambda'}) < ||\rho || \cos (\theta_{\rho,\lambda}) = \frac{\mu^\rho(v,\lambda)}{||\lambda||} \]
which shows that $\lambda$ is not $\rho$-adapted to $v$.
\end{proof}

\begin{lemma}\label{lemma on Slambda}
Let $\lambda$ be a non-trivial 1-PS of $G$ which indexes a Hesselink stratum $S_{[\lambda]}$; then
\[ S_\lambda = p_\lambda^{-1}(Z_\lambda) \]
where $p_\lambda : V^\lambda_+ \ra V^\lambda$ is the retraction given by sending $v$ to $\lim_{t \to 0} \lambda(t) \cdot v$. Moreover, $S_\lambda$ is a locally closed subvariety of $V$.
\end{lemma}
\begin{proof}
 By Theorem \ref{kempf thm}, if $v \in S_\lambda$, then $p_{\lambda}(v) \in Z_\lambda$. Now suppose $v \in V_+^\lambda$ and $p_\lambda(v) \in Z_\lambda$. If $v \notin S_{[\lambda]}$, then $v$ is $\rho$-unstable as $\mu^{\rho}(v,\lambda)=(\rho,\lambda) <0$ and so we can assume $v \in S_{[\lambda']}$. Then $p_\lambda(v) \in \partial S_{[\lambda']} \cap S_{[\lambda]}$ and so $[\lambda'] < [\lambda]$. However as both $\lim_{t \to 0} \lambda(t) \cdot v$ and $\lim_{t \to 0} \lambda'(t) \cdot v$ exist,
\begin{equation}\label{eqn F}
 M^\rho(v) =\frac{(\rho,\lambda')}{||\lambda'||} < \frac{(\rho,\lambda)}{||\lambda||}
\end{equation}
which contradicts $[\lambda'] < [\lambda]$. Hence $v \in S_{[\lambda]}\cap V_+^\lambda = S_\lambda$. We note that $Z_\lambda$ is open in $V^{\lambda}$ as it is the semistable set for a reductive group action and since also $S_\lambda = p_\lambda^{-1}(Z_\lambda)$ we have that $S_\lambda$ is locally closed. 
\end{proof}

\subsection{The indices for Hesselink's stratification}\label{hess ind}
If we fix a maximal torus $T$ of $G$, then every conjugacy class $[\lambda]$ has a representative $\lambda$ which is a 1-PS of $T$. It follows that to calculate the indices $[\lambda]$ occurring in Hesselink's stratification, we can just calculate the 1-PSs of $T$ which are $\rho$-adapted to $\rho$-unstable points for the action of $T$. This can be done combinatorially as in Lemma \ref{determine wts Hess} where we find the unique indivisible 1-PS which is adapted to an unstable point $v$ by finding the ray in the cone $C_v$ whose angle to $\rho$ is maximal. As there are only finitely many $T$-weights, there are only finitely many cones $C_v$ to consider (which can be determined combinatorially from subsets of the $T$-weights).

\section{The moment map and symplectic reduction}\label{sec sym}

Let $K$ be a compact real Lie group acting smoothly on a complex vector space $V$.
We fix a Hermitian inner product $H : V \times V \ra \CC$ on $V$ and assume that $K$ acts by unitary transformations of $V$. The imaginary part of $H$ is a symplectic inner product on $V$ which we can multiply by any nonzero real scalar. We let $\omega = \frac{1}{\pi} \text{Im} H : V \times V \ra \RR$ denote our symplectic inner product on $V$; then
\[ \omega(v,w) = \frac{1}{2 \pi i} \left ( H(v,w) - \overline{H(v,w)}\right)=\frac{1}{2 \pi i} \left ( H(v,w) - {H(w,v)}\right).\]
By identifying $TV \cong V \times V$ we see that $(V,\omega)$ is a symplectic manifold and $\omega$ is $K$-invariant. The infinitesimal action is a Lie algebra homomorphism
$\mathfrak{K}  \ra \mathrm{Vect}(V) $ given by $\alpha \mapsto \alpha_V$ where $\alpha_v :=(\alpha_V)_v$ is given by
\[\alpha_{v}= \frac{d}{dt} \exp(t\alpha) \cdot v|_{t = 0} \in T_vV\]
for $\alpha \in \mathfrak{K} =\text{Lie} K$. 

\begin{defn}
A moment map for the action of $K$ on $V$ is a smooth map $\mu : V \ra \mathfrak{K}^*$ where $\mathfrak{K} = \text{Lie} K$ such that 
\begin{enumerate}
\item $\mu$ is equivariant for the given action of $K$ on $V$ and the coadjoint action of $K$ on $\mathfrak{K}^*$.
\item For $v \in V$, $\alpha \in \mathfrak{K}$ and $\zeta \in V \cong T_vV$ we have
\[ d\mu_v (\zeta) \cdot \alpha = \omega(\alpha_v, \zeta) \]
where $\cdot : \mathfrak{K}^* \times \mathfrak{K} \ra \RR$  denotes the natural pairing, $d\mu_v$ denotes the derivative of $\mu$ at $v$ and $\alpha_v \in V \cong T_vV $ denotes the infinitesimal action of $\alpha$ on $v$. 
\end{enumerate}
\end{defn}

\begin{lemma}
For the standard representation of $\U(V)$ on $V$, there is a natural moment map $\mu : V \ra \mathfrak{u}(V)^*$ given by 
\[ \mu(v) \cdot \alpha = \frac{1}{2}\omega(\alpha v,v) =\frac{1}{2\pi i}H(\alpha v,v) \]
for $v \in V$ and $\alpha \in \mathfrak{u}(V)$. Moreover, if $K$ acts unitarily on $V$, then the composition of the above map with the projection $ \mathfrak{u}(V)^* \ra \mathfrak{K}^*$ associated to this action is a moment map for the action on $V$.
\end{lemma}
\begin{proof} We firstly note that the second equality in the definition of $\mu$ is due to the fact that
\begin{equation}\label{useful eqn for H} H(\alpha v,w) + H(v,\alpha w) = 0 \end{equation}
for all $v,w \in V$ and $\alpha \in \mathfrak{u}(V)$. The equivariance of $\mu$ follows immediately from the fact that $H$ is invariant for the action of $\U(V)$. The infinitesimal action of $\alpha \in \mathfrak{u}(V)$ at $v \in V$ is
\[ \alpha_v := \frac{d}{dt} \exp(t\alpha) \cdot v = \alpha v \in V \cong T_vV.\]
For $v \in V$, $\alpha \in \mathfrak{u}(V)$ and $\zeta \in V \cong T_vV$, we have
\begin{align*}
 d\mu (v) (\zeta)\cdot \alpha &:= \frac{d}{dt} \mu(v + t \zeta) \cdot \alpha  |_{t=0} = \frac{1}{2 }  \frac{d}{dt}  \omega(\alpha(v + t \zeta),v + t \zeta)|_{t=0} \\ &= \frac{1}{2 }\left[\omega(\alpha v,\zeta)+\omega(\alpha \zeta,v) \right] = \frac{1}{4 \pi i}\left[ H(\alpha v,\zeta) - H(\zeta, \alpha v) +H(\alpha \zeta,v) -H(v,\alpha \zeta)\right] \\
&\begin{smallmatrix}  (\ref{useful eqn for H}) \\=\end{smallmatrix}\frac{1}{2 \pi i} \left[H(\alpha v,\zeta) - H(\zeta,\alpha v) \right] =: \omega(\alpha v,\zeta)
\end{align*}
and hence $\mu$ is a moment map for the action of $\U(V)$. If $\varphi : K \ra \U(V)$ denotes the representation corresponding to the action of $K$ on $V$ and we use $\varphi : \mathfrak{K} \ra \mathfrak{u}(V)$ to also denote the associated Lie algebra homomorphism, then the natural moment map for the $K$-action on $V$ is
\[ \mu_K = \varphi^* \circ \mu_{\U(V)} : V \ra \mathfrak{u}(V)^* \ra \mathfrak{K}^* \]
where $\varphi^*$ is dual to $\varphi$. Thus
\[ \mu_K(v) \cdot \alpha =\frac{1}{2\pi i}H(\varphi(\alpha) v,v), \]
although we shall often for simplicity write $\alpha$ to mean $\varphi(\alpha)$ in such expressions.
\end{proof}

\begin{rmk}
Given a character $\rho : K \ra \U(1)$ we can identify $\U(1) \cong S^1$ and $\mathrm{Lie} S^1 \cong 2\pi i \RR$. Let $d \rho : \mathfrak{K} \ra 2 \pi i \RR$ denote the derivative of this character. Then $\frac{1}{2\pi i}d\rho \in \mathfrak{K}^*$ is fixed by the coadjoint action of $K$ on $\mathfrak{K}^*$ and so we can define a shifted moment map $\mu^\rho : V \ra \mathfrak{K}^*$ by 
\[ \mu^\rho(v) \cdot \alpha = \frac{1}{2}\omega(\alpha v,v) - \frac{1}{2\pi i} d \rho \cdot \alpha = \frac{1}{2\pi i} \left( H(\alpha v,v) - d\rho \cdot \alpha\right). \]
We shall refer to this as the natural moment map for the action of $K$ shifted by $\rho$.
\end{rmk}

In symplectic geometry, the symplectic reduction 
is used as a quotient:

\begin{defn}
The symplectic reduction for the action of $K$ on $V$ (with respect to the character $\rho$) is given by $V^{\rho\text{-red}} = (\mu^\rho)^{-1}(0)/K$.
\end{defn}

If $0$ is a regular value of $\mu^\rho$ and the action of $K$ on $V$ is proper, then the symplectic reduction  $V^{\rho-\mathrm{red}}$ is a real manifold of dimension $\dim_\RR V - 2 \dim_\RR K$. Moreover, Marsden and Weinstein \cite{mw} and Meyer \cite{meyer} show that there is a canonical induced symplectic form $\omega^{\mathrm{red}}$ on $V^{\rho\text{-red}}$ such that $i^* \omega = \pi^* \omega^{\text{red}}$ where $i : (\mu^\rho)^{-1}(0) \hookrightarrow V$ is the inclusion and $\pi :  (\mu^\rho)^{-1}(0) \ra V^{\rho\text{-red}} $ is the projection.

\subsection{The norm square of the moment map}

Given an inner product $(-,-)$ on the Lie algebra $\mathfrak{K}$ of $K$ which is invariant under the adjoint action of $K$ on $\mathfrak{K}$, we can identify $\mathfrak{K} \cong \mathfrak{K}^*$. We let $|| - ||$ denote the associated norm on $\mathfrak{K} \cong \mathfrak{K}^*$ given by $||\alpha||^2 = (\alpha,\alpha)$.

Suppose we have an action of $K$ on a complex affine space $V$ (by unitary transformations) with natural moment map $\mu:=\mu^\rho : V \ra \mathfrak{K}^*$ shifted by a character $\rho$ as above. We let $\mu^* : V \ra \mathfrak{K}$ denote the map constructed from $\mu$ and the isomorphism $\mathfrak{K}^* \cong \mathfrak{K}$; thus
\[ \mu(v) \cdot \alpha =(\mu^*(v),\alpha).\]
We consider the norm square of the moment map
\[  \begin{array}{ll} || \mu ||^2 : & V \ra \RR \\ & v \mapsto ||\mu(v)||^2 \end{array}\]
which is a real analytic function. The derivative of this function is a 1-form $d ||\mu||^2$ on $V$ which at $v \in V$ and $\zeta \in V \cong T_vV$ is given by
\begin{align}\label{grad} 
 d||\mu||^2_v(\zeta) &:= \frac{d}{dt} || \mu(v + t \zeta )||^2 |_{t = 0}  = 2(d\mu_v(\zeta),\mu(v))  \\&= 2d\mu_v(\zeta) \cdot \mu^*(v) =2\omega(\mu^*(v)_v,\zeta). \notag
\end{align}
The critical locus $\text{crit} || \mu ||^2 := \{ v : d||\mu||^2_v = 0 \}$ is a closed subset of $V$.

\begin{lemma}\label{crit}
Let $v \in V$; then the following are equivalent:
\begin{enumerate}
\item $v$ is a critical point of $||\mu||^2 : V \ra \RR$.
\item The infinitesimal action of $\mu^*(v)$ at $v$ is zero:  $\mu^*(v)_v  = 0$.
\item $\mu^*(v) \in \mathfrak{K}_v$ where $\mathfrak{K}_v$ is the Lie algebra of the stabiliser $K_v$ of $v$. 
\end{enumerate} 
\end{lemma}

Clearly $\mu^{-1}(0) = (\mu^*)^{-1}(0)$ is a subset of the critical locus $\text{crit}||\mu||^2$; we refer to these critical points as minimal critical points. For $\beta \in \mathfrak{K}$, we can consider the function $\mu_\beta : V \ra \RR$ given by $\mu_\beta(v) = \mu(v) \cdot \beta$. As $d\mu_\beta = \omega(\beta_V,-)$, the critical locus of $\mu_\beta$ is the set of points $v \in V$ for which the infinitesimal action of $\beta$ is zero; that is, \[ \text{crit} \mu_\beta = \{ v \in V : \beta_v = 0\}.\] It follows from Lemma \ref{crit} above, that the intersection $ \text{crit}\mu_\beta \cap (\mu^*)^{-1}(\beta)$ is contained in the critical locus $\text{crit}||\mu||^2$. Since $||\mu||^2 : V \ra \RR$ is $K$-invariant, the critical locus $\text{crit}||\mu||^2$ is invariant under the action of $K$. Therefore the disjoint sets
\[ C_{K \cdot \beta}:= K( \text{crit}\mu_\beta \cap (\mu^*)^{-1}(\beta)) =  \text{crit} ||\mu||^2 \cap (\mu^*)^{-1}(K \cdot \beta)\]
indexed by adjoint orbits $K \cdot \beta$ cover the critical locus of $||\mu||^2$ computed at (\ref{grad}) and we refer to these subsets $C_{K \cdot \beta}$ as the critical subsets.

We take a metric compatible with the complex and symplectic structure on $V$ and this defines a duality between 1-forms and vector fields where $\omega(-,\psi)$ is dual to $i \psi$. The 1-form $d||\mu||^2$ is dual (under this metric) to a vector field on $V$ which we denote by $\text{grad} ||\mu||^2$ and from the calculation of $d||\mu||^2$ at (\ref{grad}), we have at $v \in V$ that
\begin{equation}\label{grad2}
  \text{grad}||\mu||^2_v=-2i\mu^*(v)_v .
\end{equation}

\begin{defn}
For $v \in V$, the negative gradient flow of $||\mu||^2$ at $v$ is a path $\gamma_v(t)$ in $V$ which is defined in some open neighbourhood of $0 \in \RR_{\geq 0}$ and satisfies
\[ \gamma_v(0) = v \quad \gamma_v'(t) = - \mathrm{grad} ||\mu||^2_{\gamma_v(t)} . \]
\end{defn}

The negative gradient flow exists by the local existence of solutions to ordinary differential equations and it is well known that $\gamma_v(t)$ exists for all real $t$ (for example, see \cite{jost}). If $v$ is a critical point of $||\mu||^2$, then $\text{grad}||\mu||^2_v = 0$ and so $\gamma_v(t) = v$ is the constant path. From the expression for $\text{grad}||\mu||^2_v $ given at (\ref{grad2}), we see that the finite time negative gradient flow is contained in the orbit of $v$ under the action of the complexified group $K_\CC$ whose Lie algebra $\mathfrak{K}_\CC \cong \mathfrak{K} \oplus i \mathfrak{K}$ is the complexification of $\mathfrak{K}$. Moreover, even though we are working in a (non-compact) affine space, the gradient $\gamma_v(t)$ converges to a critical value $v_\infty$ of $||\mu||^2$ by work of Harada and Wilkin \cite{haradawilkin} and Sjamaar \cite{sjamaar}. 

\subsection{Morse-theoretic stratification}

We use the negative gradient flow of $||\mu||^2$ to produce a Morse-theoretic stratification of $V$ as follows. The index set for the stratification is a finite number of adjoint orbits which index nonempty critical subsets:
\begin{equation}\label{index set}
 \mathcal{B} := \{ K \cdot \beta : \beta \in \mathfrak{K} \text{ and } C_{K \cdot \beta} \neq \emptyset\}.
\end{equation}
For $K \cdot \beta \in \mathcal{B}$, we define the associated stratum 
\[ S_{K \cdot \beta} := \{ v \in V : \lim_{t \to \infty} \gamma_v(t) \in C_{K \cdot \beta} \} \]
to be the set of points whose negative gradient flow converges to $C_{K \cdot \beta}$. We refer to the stratification
\begin{equation}\label{morse strat}
  V = \bigsqcup_{K \cdot \beta \in \mathcal{B}} S_{K \cdot \beta} 
\end{equation}
as a Morse stratification; the partial order on the indices is given by using the $K$-invariant norm $|| - ||$ on $\mathfrak{K}$ in the natural way. We say one stratum is lower or higher than another if its corresponding index is with respect to this norm. In particular, the lowest stratum is indexed by $0 \in \mathfrak{K}$ and $S_0 \subset V$ is open with limit set $C_0 = \mu^{-1}(0)$. 

If we fix a maximal torus $T \subset K$ and a positive Weyl chamber $\mathfrak{t}_+ \subset \mathfrak{t} = \text{ Lie } T$, then an adjoint orbit $K \cdot \beta$ meets $\mathfrak{t}_+$ in a single point $\beta$ and so we can view the index set $\mathcal{B}$ as a finite set of elements in $\mathfrak{t}_+$. 

\subsection{Description of the indices}\label{sec desc indices}

The indices $\beta \in \mathcal{B}$ can be computed from the weights of a maximal torus $T$ of $K$ acting on $V$ similarly to the projective case given in \cite{kirwan,ness}, although the role of weight polytopes is now played by shifted weight cones. A torus action gives an orthogonal decomposition $V = \oplus_\chi V_\chi$ where the indices $\chi$ are characters of $T$ and $V_\chi = \{ v \in V : t \cdot v = \chi(t) v  \text{ for all } t \in T \}$. We use the natural identification $\chi^*(T)_\RR \cong \mathfrak{t}^*$ given by \[\chi \mapsto \frac{1}{2\pi i}d\chi\] to identify (real) characters with elements of $\mathfrak{t}^* \cong \mathfrak{t}$ and will write $\chi$ to mean either a character or an element of $\mathfrak{t}^* \cong \mathfrak{t}$. We refer to the points in the lattice $\chi^*(T) \subset \mathfrak{t}^*$ as integral points. Let $T\text{-wt}:=\{ \chi : V_\chi \neq 0 \}$ denote the set of $T$-weights for the action. The advantage of working with a maximal torus is that the image of a complex vector space under a moment map for a torus action (shifted by $\rho$) is a cone generated by the torus weights (shifted by $-\rho$) \cite{atiyah,gsmmapconv}. 

If $B$ is a (possibly empty) subset of the $T$-weights, then we associate to $B$ a cone $C(B)$ in $\mathfrak{t}$ given by the positive span of these weights (where the cone associated to the empty set is the origin $0$). Let $C_{-\rho}(B):= C(B) - \rho$ denote the translation of this cone by $-\rho$; then we define $\beta(B)$ to be the closest point of $C_{-\rho}(B)$ to the origin. We claim that 
\[ \mathcal{B} = \{ \beta(B) \in \mathfrak{t}_+ : B \subset T\text{-wt} \}. \]
The proof follows in exactly the same way as the projective case (for example, see \cite{kirwan} $\S$3) except that the weight polytopes are replaced by our shifted weight cones: to show $\beta(B)$ is an index one can find $v \in V_B = \oplus_{\chi \in B} V_\chi$ such that $v \in S_{K \cdot \beta(B)}$ and given $\beta \in \mathcal{B}$ one can find $v \in S_v$ such that $\mu^*(v) \in \mathfrak{t}$ and then $\beta = \beta(B)$ where $B= \{ \chi : v_\chi \neq 0 \} $ and $v = \sum v_\chi$. 

From this description we see that the indices $\beta$ (which are the closest points to the origin of cones defined by integral weights $\chi$ shifted by an integral weight $\rho$) are rational weights; that is, $n\beta \in \chi^*(T) \cong \chi_*(T) \cong \ZZ^{\rk T}$ for some natural number $n$.

\begin{rmk}
If $\rho$ is the trivial character, then any cone $C_0(B)$ associated to a subset $B$ of the torus weights will contain $0$ and so $\beta(B) = 0$ for all such subsets $B$. Thus if $\rho = 0$, there is only one Morse stratum $S_0 = V$.
\end{rmk}

\subsection{An alternative description of the strata}\label{sec mubeta}

In this section we given an alternative description of the Morse strata which is similar to the description in the projective setting given by Kirwan in \cite{kirwan} $\S$6. We note that our notation differs slightly from that given in \cite{kirwan}.

For $\beta \in \mathfrak{K}$, we let $\mu_\beta : V \ra \RR$ be given by $\mu_\beta(v) = \mu(v) \cdot \beta$ as above. Then $d\mu_\beta = \omega(\beta_V,-)$ and \[\text{crit} \mu_\beta = \{ v \in V: \beta \in \mathfrak{K}_v\} =:V^{\beta}. \]
We note that $\mu_\beta$ is constant on this critical locus: $\mu_\beta(V^{\beta}) = - d \rho \cdot \beta / 2 \pi i$ (in the projective setting this is not always the case and so one only considers certain connected components of this critical locus). The gradient vector field associated to $\mu_\beta$ is $\text{grad} \mu_\beta = -i \beta_V $ and so the negative gradient flow of $\mu_\beta$ at a point $v \in V$ is given by
\[ \gamma^{\beta}_v(t) = \exp(it \beta) \cdot v.\]
We let $V^{\beta}_+ := \{ v \in V : \lim_{t \to \infty} \exp(it \beta) \cdot v \text{ exists} \}$ denote the subset of points for which the limit of the negative gradient flow exists; then there is a retraction  $p_\beta : V_+^\beta \ra V^\beta$ given by the negative gradient flow. In the projective setting, the negative gradient flow of any point under $\mu_\beta$ converges and the notation $p_\beta : Y_\beta \ra Z_\beta$ is used.

Let $K_\beta$ denote the stabiliser of $\beta$ under the adjoint action of $K$ and let $\mathfrak{K}_\beta$ denote the associated Lie algebra. Then $V^{\beta}$ is invariant under the action of $K_\beta$ and for $v \in V^{\beta}$ we note that $\mu(v) \in \mathfrak{K}_\beta^*$ (for example, in \cite{ness} see the proof of Theorem 9.2). Hence $\mu : V^\beta \ra \mathfrak{K}^*_\beta$ is a moment map for the action of $K_\beta$ on $V^\beta$. As $\beta^* = (\beta , -) \in \mathfrak{K}^*_\beta$ is a central element, $\mu - \beta^*$ is also a moment map for the action of $K_\beta$ on $V^\beta$. If we wish to emphasise that we are considering $\mu$ as a moment map for the action of $K_\beta$ on $V^\beta$ or $K$ on $V$ we will write $\mu_{K_\beta}$ or $\mu_K$ respectively.

The critical subset of $||\mu||^2$ associated to $\beta$ (or more precisely the adjoint orbit of $\beta$) is
\[ C_{K \cdot \beta}:= K(\text{crit}\mu_\beta \cap \mu^{-1}(\beta^*)) =  \text{crit} ||\mu||^2 \cap \mu^{-1}(K \cdot \beta^*) \]
and we let $C_\beta :=  \text{crit} \mu_\beta \cap \mu^{-1}(\beta^*) = V^{\beta} \cap \mu^{-1}(\beta^*) \subset C_{K \cdot \beta}$. 

\begin{lemma}\label{lemma H}
 For $\beta \in \mathfrak{K}$, the set $C_\beta$ is equal to the set of minimal critical points for the norm square of the moment map $\mu-\beta^* : V_\beta \ra \mathfrak{K}_\beta^*$; that is,
\[ C_\beta = (\mu-\beta^*)^{-1}(0). \]
\end{lemma}
\begin{proof}
We verify that
$(\mu_{K_\beta} -\beta^*)^{-1}(0) = \mu_{K_\beta}^{-1}(\beta^*) = \mu_K^{-1}(\beta^*) \cap V_{\beta} = C_\beta. $
\end{proof}

\begin{rmk} \label{norm of beta}
 If $C_\beta$ is nonempty, then $\mu_\beta (C_\beta) = || \beta ||^2$ but as $\beta_v = 0$ for $v \in V^{\beta}$, we also have $\mu_\beta(V^\beta) =- d \rho \cdot \beta / 2 \pi i$. Therefore the norm of $\beta$ satisfies
\[ ||\beta||^2 = - d \rho \cdot \beta / 2\pi i. \]
\end{rmk}

Then $||\mu- \beta^*||^2$ is used to obtain a description of the Morse stratum $S_{K \cdot \beta}$ for $||\mu||^2$ as in \cite{kirwan}. 

\begin{defn} Let $Z_\beta^{\min} \subset V^\beta$ be the minimal Morse stratum for the function $||\mu- \beta^*||^2$ on $V^\beta$ (whose points flow to $C_\beta$ by Lemma \ref{lemma H} above) and $Y_\beta^{\min} \subset V^\beta_+$ be the preimage of $Z_\beta^{\min}\subset V^\beta$ under $p_\beta : V^{\beta}_+ \ra V^{\beta}$. \end{defn}

As in \cite{kirwan}, we can consider a parabolic subgroup $P_\beta$ of the complexified group $G := K_\CC$
\[ P_\beta = \{ g \in G : \lim_{t \to \infty} \exp(it\beta)g\exp(it\beta)^{-1} \text{ exists} \} \]
which leaves $V^\beta_+$ and $Y_\beta^{\min}$ invariant. As $G = K P_\beta$, we have that $G V_+^\beta = K V_+^\beta$ and similarly $G Y_\beta^{\min}= K Y_\beta^{\min}$. In Theorem \ref{thm alt desc} below, we prove that $S_{K \cdot \beta} = G Y_\beta^{\min}$.

\section{A comparison of the algebraic and symplectic descriptions}\label{sec corr}

We recall that a group $G$ is complex reductive if and only if it is the complexification of its maximal compact subgroup $K \subset G$. In this section we suppose that we have a linear action of a complex reductive group $G = K_\CC$ on a complex vector space $V$ for which $K$ acts unitarily (with respect to a fixed Hermitian inner product $H$ on $V$). We linearise the action as in $\S$\ref{sec GIT} by choosing a character $\rho : G \ra \CC^*$ and use this to construct a linearisation $L_\rho$ whose underlying line bundle is the trivial line bundle $L = V \times \CC$. We assume that $\rho(K) \subset \U(1) \cong S^1$ so that we can consider the restriction of $\rho$ to $K$ as a compact character. Then the main results of this section are:

\medskip

i) The GIT quotient $V /\!/_\rho G$ with respect to $\rho$ is homeomorphic to the symplectic reduction $V^{\rho\text{-red}} = (\mu^\rho)^{-1}(0)/K$ with respect to $\rho$ (cf. Theorem \ref{aff KN} below). 

\medskip

ii) If we fix a $K$-invariant inner product on $\mathfrak{K}$, then Hesselink's stratification of $V$ by $\rho$-adapted 1-PSs agrees with the Morse stratification of $V$ associated to $||\mu^\rho||^2$ (cf. Theorem \ref{Hess is Morse} below).

\subsection{Affine Kempf--Ness theorem with respect to a character}

Before we can state the affine Kempf--Ness theorem, we need one additional definition from GIT. We recall that the action is lifted to $L^{-1} :=L^{-1}_\rho= V \times \CC$ by using the character $-\rho$.

\begin{defn}
An orbit $G \cdot v$ is said to be $\rho$-polystable if $G \cdot \tilde{v}$ is closed in $L^{-1}$ where $\tilde{v} =(v,a) \in L^{-1}$ and $a \neq 0$.
\end{defn}

One can show that an orbit is $\rho$-polystable if and only if it is $\rho$-semistable and closed in $ V^{\rho\text{-ss}}$. By Lemma \ref{top crit}, we have inclusions
\[ V^{\rho\text{-s}} \subset V^{\rho\text{-ps}} \subset V^{\rho\text{-ss}} \]
where $V^{\rho\text{-ps}}$ denotes the locus of $\rho$-polystable points. From now on we use the terms stability, polystability and semistability to all mean with respect to the character $\rho$. As the closure of every semistable orbit in the semistable locus contains a unique closed orbit (which is polystable), the GIT quotient $V/\!/_\rho G$ is topologically the orbit space $V^{\rho\text{-ps}}/G$. 

\begin{thm}\label{aff KN}(Affine Kempf--Ness theorem)  Let $G = K_\CC$ be a complex reductive group acting linearly a complex vector space $V$ and suppose $K$ acts unitarily with respect to a fixed Hermitian inner product on $V$. Given a character $\rho$ for this action, let $\mu:= \mu^\rho : V \ra \mathfrak{K}^*$ denote the moment map for this action (shifted by a character $\rho$). Then
\begin{enumerate}
\renewcommand{\labelenumi}{\roman{enumi})}
\item A $G$-orbit meets the preimage of $0$ under $\mu$ if and only if it is polystable; that is, $G \mu^{-1}(0) = V^{\rho\text{-}\mathrm{ps}}$.
\item A polystable $G$-orbit meets $\mu^{-1}(0)$ in a single $K$-orbit.
\item A point $v \in V$ is semistable if and only if its $G$-orbit closure meets $\mu^{-1}(0)$.
\item The lowest Morse stratum $S_0$ agrees with the GIT semistable locus $V^{\rho\text{-}\mathrm{ss}}$.
\item The inclusion $\mu^{-1}(0) \subset V^{\rho\text{-}\mathrm{ss}}$ induces a homeomorphism
\[ V^{\rho\text{-}\mathrm{red}} := \mu^{-1}(0)/K \cong V/\!/_\rho G. \]
\end{enumerate}
\end{thm}

We note that King proves the first two parts of this theorem in \cite{king} $\S$6, although we shall also provide a proof here in part for completeness and also in part as it allows us to state some more general results which we shall need for the proof of Theorem \ref{Hess is Morse}.

To prove the above theorem we introduce a function $p_v : \mathfrak{K} \ra \RR$ for each $v \in V$:
\[ p_v(\alpha) = \frac{1}{4\pi}H(\exp(i\alpha)v, \exp(i\alpha)v) + \frac{1}{2\pi i} d\rho \cdot \alpha \]
where we write $\exp(i\alpha)v$ to mean the action of $\exp(i\alpha) \in G$ on $v \in V$.

\begin{lemma}\label{lemma A}
An element $\alpha \in \mathfrak{K}$ is a critical point of $p_v$ if and only if $\mu(\exp(i\alpha)v) = 0$. Moreover, the second derivatives of $p_v$ are non-negative and so every critical point is a minimum.
\end{lemma}
\begin{proof} The first and second order derivatives of $p_v$ are:
\begin{align*}
d(p_v)_\alpha(\beta) &:= \frac{d}{dt} p_v(\alpha + t \beta)|_{t=0} \\ & = \frac{1}{4\pi} \left[H(i\beta\exp(i\alpha)v, \exp(i\alpha)v) +H(\exp(i\alpha)v, i\beta\exp(i\alpha)v) \right] +\frac{1}{2\pi i} d\rho \cdot \beta
\\ & = \frac{i}{4\pi } \left[H(\beta\exp(i\alpha)v, \exp(i\alpha)v) -H(\exp(i\alpha)v, \beta\exp(i\alpha)v) \right] +\frac{1}{2\pi i} d\rho \cdot \beta
\\ & \begin{smallmatrix}  (\ref{useful eqn for H}) \\=\end{smallmatrix} \frac{i}{2\pi} H(\beta\exp(i\alpha)v, \exp(i\alpha)v)  +\frac{1}{2\pi i} d\rho \cdot \beta = - \mu(\exp(i\alpha)v) \cdot \beta \end{align*}
and
\begin{align*}
d^2(p_v)_\alpha(\beta) &:= \frac{d^2}{dt^2} p_v(\alpha + t \beta)|_{t=0} = \frac{d}{dt} \left(\frac{i}{2\pi} H(\beta\exp(i\alpha + t \beta)v, \exp(i\alpha)v) +\frac{1}{2\pi i} d\rho \cdot \beta\right)|_{t = 0}
\\ & = \frac{-1}{2\pi} \left[H(\beta^2\exp(i\alpha)v, \exp(i\alpha)v) -H(\beta\exp(i\alpha)v, \beta\exp(i\alpha)v) \right] 
\\ & \begin{smallmatrix}  (\ref{useful eqn for H}) \\=\end{smallmatrix}\frac{1}{\pi}H(\beta\exp(i\alpha)v, \beta\exp(i\alpha)v) \geq 0. 
\end{align*}
Then the lemma follows immediately from these calculations.
\end{proof}

\begin{prop}\label{prop B}
Let $v \in V$; then $p_v$ has a minimum if and only if $v$ is polystable.
\end{prop}
\begin{proof}
If $\alpha$ is a critical point of $p_v$, then $u=\exp(i\alpha)v \in \mu^{-1}(0)$ by Lemma \ref{lemma A}. Given a 1-PS $\lambda$ of $G$ we can consider the limit
\begin{equation}\label{lim for u} \lim_{t \to 0} \lambda(t) \cdot (u,1) \end{equation}
in $L^{-1} =V \times \CC$. If the limit exists, then we can assume (by conjugating $\lambda$ by $g$ and replacing $u$ by $gu$ if necessary) that $\lambda(S^1) \subset K$. By Lemma \ref{lemma D} below, if 
\[ \beta = d\lambda(2 \pi i) = \frac{d}{dt} \lambda(\exp(2\pi i t))|_{t=0} \in \mathfrak{K}, \]
then $\mu(u) \cdot \beta \geq 0$ with equality if and only if $\lambda$ fixes $(u,1)$. As $\mu(u) = 0$, we conclude that $\lambda$ fixes $(u,1)$ and so it follows from Theorem \ref{fund thm} that $G \cdot (u,1)$ is closed in $L^{-1}$; i.e. $G \cdot u = G \cdot v$ is polystable.

Conversely, if $G \cdot v$ is polystable then $G\cdot(v,1)$ is closed in $L^{-1}$. It follows from this that for $\alpha \in \mathfrak{K}$:
\begin{enumerate} 
\item If $\exp(i\alpha) \in G_v$, then $\frac{1}{2\pi i}d \rho \cdot \alpha = 0$.
\item If $\exp(i\alpha) \notin G_v$, then either $\lim_{t \to \infty} \exp(it\alpha) \cdot v$ does not exist or $\frac{1}{2\pi i}d \rho \cdot \alpha > 0$.
\end{enumerate}
Hence the convex function $p_v : \mathfrak{K} \ra \RR$ is bounded below and achieves a minimum. 
\end{proof}

\begin{lemma}\label{lemma D}
 Let $v \in V$ and $\lambda : \CC^* \ra G$ be a group homomorphism such that $\lambda(S^1) \subset K$ and let $ V = \oplus_{r \in \ZZ} V_r$ be the weight decomposition associated to the action of $\lambda(\CC^*)$ on $V$; then
\[ \mu(v) \cdot \alpha = \sum_{r \in \ZZ} r H(v_r,v_r) - (\rho,\lambda)\]
where $v = \sum_r v_r$ is written with respect to the above weight decomposition and
\[ \alpha = d\lambda(2 \pi i) =\frac{d}{dt} \lambda(\exp(2\pi i t))|_{t=0} \in \mathfrak{K}. \]
Furthermore,
\begin{enumerate} \renewcommand{\labelenumi}{\roman{enumi})}
 \item If $\lim_{t \to 0} \lambda(t) \cdot v$ exists in $V$, then $\mu(v) \cdot \alpha \geq -(\rho,\lambda)$ with equality if and only if $\lambda$ fixes $v$.
 \item If $\lim_{t \to 0} \lambda(t) \cdot (v,1)$ exists in $L^{-1} =V \times \CC$, then $\mu(v) \cdot \alpha \geq 0$ with equality if and only if $\lambda$ fixes $(v,1)$.
\end{enumerate}
\end{lemma}
\begin{proof}
By definition $V_r := \{ v \in V : \lambda(s) \cdot v = t^r v \text{ for all } t \in \CC^*\}$, and so the infinitesimal action of $\alpha \in \mathfrak{K}$ on $v_r \in V_r$ is $\alpha_{v_r} =\alpha v_r = 2 \pi i r v_r$. As $\lambda$ is a 1-PS of $K$, the corresponding weight decomposition $V = \oplus V_r$ is orthogonal with respect to the Hermitian inner product $H$ and so $H(v_r,v_s) = 0$ for $s \neq r$. Therefore 
\begin{align*}
\mu(v) \cdot \alpha & = \frac{1}{2 \pi i} \left(H(\alpha v,v) - d \rho \cdot \alpha \right)  =
 \frac{1}{2 \pi i} \left( \sum_{r,s} H(2 \pi i r v_r,v_s) - d \rho \cdot \alpha\right) \\ & =  \sum_r r H(v_r,v_r) - (\rho,\lambda)
\end{align*}
where $ 2 \pi i (\rho,\lambda) = d\rho \cdot \alpha$ as $\alpha = d\lambda(2 \pi i)$. As $\lim_{t \to 0} \lambda(t) \cdot v$ exists, then this implies $v_r = 0$ for all $r <0$. Hence $\sum_r r H(v_r,v_r) \geq 0$ with equality if and only if $v = v_0$ which proves part i). If $\lim_{t \to  0} \lambda(t) \cdot (v,1)$ exists, then $v_r = 0$ for all $r <0$ and $(\rho,\lambda) \leq 0$. Hence part ii) follows and we have equality if and only if $v = v_0$ and $(\rho ,\lambda) = 0$; that is, $(v,1)$ is fixed by $\lambda$.
\end{proof}

We just need to state one additional lemma before we can prove the affine version of the Kempf--Ness theorem.

\begin{lemma}\label{lemma C}
Let $v \in V$ be a critical point of $||\mu||^2$ where $|| - ||$ is the norm associated to a $K$-invariant inner product on $\mathfrak{K}$; then $G \cdot v$ meets the critical locus in precisely the $K$-orbit of $v$.
\end{lemma}
\begin{proof}
If $w = g v$ is also critical for $||\mu||^2$, then we want to show that $w \in K \cdot v$. Since $G = K \exp(i\mathfrak{K})$, it suffices to prove this when $g = \exp(i \alpha)$ for some $\alpha \in \mathfrak{K}$. In this case, both $v$ and $w$ belong to some critical subset $C_{K \cdot \beta}$ and so there exists $k \in K$ such that
\[ \beta := \mu^*(v) = k \cdot \mu^*(w) = \mu^*(k w).\]
As $w \in K \cdot v$ if and only if $k w \in K \cdot v$, we can assume $ \beta = \mu^*(v) = \mu^*(w). $
Then, following \cite{kirwan} Lemma 7.2, we consider the function $h : \RR \ra \RR$ defined by $h(t) = \mu (\exp(it\alpha) v) \cdot \alpha$. Since 
\[ h(0) = \mu(v) \cdot \alpha = (\beta, \alpha) =  \mu(v) \cdot \alpha = h(1),\]
there exists $t \in (0,1)$ such that $h'(t) = 0$. Then
\begin{align*}
0=h'(t) &= \frac{d}{dt} \frac{1}{2i}\left( H(\alpha \exp(i t \alpha) v,\exp(i t \alpha) v) + d \rho \cdot \alpha  \right) \\ & = \frac{1}{2} [ H(\alpha^2 \exp(i t \alpha) v,\exp(i t \alpha) v) - H(\alpha \exp(i t \alpha) v,\alpha\exp(i t \alpha) v)]
\\ & \begin{smallmatrix}  (\ref{useful eqn for H}) \\ =\end{smallmatrix} -H(\alpha \exp(i t \alpha) v,\alpha\exp(i t \alpha) v). 
\end{align*}
Hence $\alpha \exp(it \alpha)  v= 0$ and so $\alpha$ is contained in $\mathfrak{K}_{\exp(it \alpha) v}$. Therefore, $\alpha \in \mathfrak{K}_v$ and $i \alpha \in \mathfrak{g}_v$, so that $w = \exp(i\alpha) v = v \in K \cdot v$ as required.
\end{proof}

We now given the proof of {Theorem \ref{aff KN}} (Affine Kempf--Ness theorem): \begin{proof}
i) follows immediately from Lemma \ref{lemma A} and Proposition \ref{prop B} and ii) follows from Lemma \ref{lemma C}. For iii), if $G \cdot v$ is semistable then there is a unique closed orbit $G \cdot w$ in $\overline{G \cdot v}$ which is polystable and so by i) $G \cdot w$ meets $\mu^{-1}(0)$. Conversely, if $w \in \mu^{-1}(0) \cap \overline{G \cdot v}$, then $w$ is polystable by i). It then follows from the openness of $V^{\rho\text{-ss}} $ in $ V$ that $v$ is also semistable.

For iv), if $v \in S_0$, then the limit of its negative gradient flow $ \lim_{t \to \infty} \gamma_v(t) $ is contained in $C_0 \cap \overline{G \cdot v}$ and as $C_0 = \mu^{-1}(0)$ by iii) we conclude that $v$ is semistable. Conversely, if $G \cdot v$ is semistable, then there is a polystable orbit $G \cdot w \subset \overline{G \cdot v}$ and by i) we have $G \cdot w$ meets $\mu^{-1}(0)$. Therefore $w \in S_0$ and as $S_0 $ is open in $V$, also $v \in S_0$.

For v), it is clear from parts i) and ii) that we have a bijection of sets
\[ \mu^{-1}(0)/K \cong V^{\rho\text{-ps}}/G \cong V/\!/_\rho G.\]
The inclusion $\mu^{-1}(0) \subset V^{\rho\text{-ss}}$ induces a continuous bijection
\[ \mu^{-1}(0)/K \ra V/\!/_\rho G\]
and the inverse of this map is constructed by using the gradient flow $V^{\rho\text{-ss}} =S_0 \ra C_0 = \mu^{-1}(0)$ associated to $||\mu||^2$ with respect to a $K$-invariant norm on $\mathfrak{K}$.
\end{proof}

\subsection{A comparison of the stratifications}

We now fix an inner product $( -, -)$ on $\mathfrak{K}$ which is invariant under the adjoint action of $K$ on its Lie algebra $\mathfrak{K}$. We use $||-||$ to denote the associated $K$-invariant norm on $\mathfrak{K}$ and also the norm on the set of conjugacy classes of 1-PSs. We assume that $||-||^2$ takes integral values on integral weights in $\mathfrak{K}$ (i.e. if $\lambda : S^1 \ra K$ is a group homomorphism, then $||d\lambda(2 \pi i)||^2 \in \ZZ$ as required by Assumption \ref{int ass}). We want to compare Hesselink's stratification of $V$ by $\rho$-adapted 1-PSs of $G=K_\CC$
\[ V = \bigsqcup_{[\lambda]} S_{[\lambda]}, \]
which is indexed by conjugacy classes of homomorphisms $\lambda : \CC^* \ra G = K_\CC$, with the Morse stratification of $V$ associated to $||\mu||^2$
\[ V = \bigsqcup_{K \cdot \beta} S_{K \cdot \beta}, \]
which is indexed by adjoint orbits in $\mathfrak{K}$. If $\rho$ is the trivial character, then Hesselink's stratification of $V$ and the Morse stratification both consist of a single stratum indexed by the trivial homomorphism $\lambda : \CC^* \ra G$ and $0 \in \mathfrak{K}$ respectively. Therefore, we assume $\rho$ is a non-trivial character.

For $v \in V$ and a 1-PS $\lambda$ of $G$ such that $\lim_{t \to 0}\lambda(t) \cdot v$ exists, in $\S$\ref{sec GIT}  we defined the Hilbert--Mumford function
\[ \mu(v,\lambda) := (\rho,\lambda) \] 
where we have now omitted the superscript $\rho$ as this character is fixed. We recall that $v \in V$ is strongly stable if there are no 1-PSs of $G$ for which $\lim_{t \to 0}\lambda(t) \cdot v$ exists. For every $v \in V$ which is not strongly stable, we define
\[ M(v) = \inf_\lambda \frac{\mu(v,\lambda)}{ ||\lambda||}\]
where the infimum is taken over all 1-PSs $\lambda$ of $G$ such that $\lim_{t \to 0}\lambda(t) \cdot v$ exists. For strongly stable points $v$, we define $M(v) := 0$.

Given a 1-PS $\lambda$ of $G$ such that $\lambda(S^1) \subset K$, we let
\[ \alpha = \alpha(\lambda) = d\lambda(2 \pi i) :=\frac{d}{dt} \lambda(\exp(2 \pi i t) )|_{t = 0} \in \mathfrak{K}\]
denote the associated integral weight. Conversely, given $\alpha \in \mathfrak{K}$, we define an associated real 1-PS subgroup $\exp(\alpha\RR)$ of $K$. This correspondence defines an isomorphism between real 1-PSs of $K$ and $\mathfrak{K}$. We recall that $\alpha$ is said to be integral if the map $\lambda : S^1 \ra K$ defined by $\lambda(\exp(2 \pi i \RR) ): = \exp(\alpha\RR)$ is a group homomorphism. For any $\alpha$, we can find a positive real number $c$ such that $c \alpha$ is integral and then we shall refer to the 1-PS $\lambda$ associated to $c\alpha$ as a 1-PS associated to $\alpha$.

\begin{lemma}\label{lemma E}
 Let $\lambda$ be a 1-PS of $G$ such that $\lambda(S^1) \subset K$ and let $\alpha\in \mathfrak{K}$ be the associated integral weight. Then for $v \in V$, we have:
\begin{enumerate} \renewcommand{\labelenumi}{\roman{enumi})}
 \item If $\lim_{t \to 0} \lambda(t) \cdot v$ exists then $\mu_\alpha(v) \geq -(\rho,\lambda) $ with equality if and only if $\lambda(t) \subset G_v$. In this case \[\frac{\mu(v,\lambda)}{ ||\lambda||} \geq -\frac{\mu_\alpha(v)}{||\alpha||} \geq - ||\mu(v)||. \]
\item Moreover, $M(v) \geq -||\mu(v)||$ with equality if and only if $v$ is a critical point of $\mu_\beta$ where $\beta = \mu^*(v)$.
\end{enumerate}
\end{lemma}
\begin{proof}
The first statement of part i) follows immediately from Lemma \ref{lemma D} i). The second statement of part i) follows from the first and the fact that
\[ \frac{\mu_\alpha(v)}{||\alpha||} = \left(\mu^*(v),\frac{\alpha}{||\alpha||}\right) \leq ||\mu(v)|| \]
with equality if and only if $\mu^*(v)$ is a positive scalar multiple of $\alpha$. 

For ii) if $v$ is strongly stable then $M(v) = 0 \geq - ||\mu(v)||$ with equality if and only if $\mu(v) = 0$. If $\mu(v) = 0$, then $v$ is critical for $\mu_0 =0$ where $0 =\mu^*(v)$. If $\beta = \mu^*(v) \neq 0$ and the strongly stable point $v$ is critical for $\mu_\beta$, then this would contradict Theorem \ref{aff KN} iv): $V^{ss} = S_0$. 

If $v$ is not strongly stable, then the inequality in ii) follows from i).  If $v$ is a critical point of $\mu_\beta$ where $\beta = \mu^*(v)$, then as $\beta$ is rational (cf. $\S$\ref{sec desc indices}) we know that $n \beta$ is integral for some positive integer $n$. The 1-PS $\lambda$ associated to $n\beta$ fixes $v$ (by definition of $v$ being a critical point of $\mu_\beta$). Hence
\[ M(v) \leq \frac{(\rho,\lambda)}{||\lambda||} = -\frac{\mu_{n\beta}(v)}{||n\beta||} = -\left(\beta,\frac{\beta}{||\beta||}\right) = -||\beta|| = -||\mu(v)|| \]
and so $M(v) = - ||\mu(v)||$. Conversely, if $M(v) = -||\mu(v)||$, then by i) there is a 1-PS $\lambda$ (corresponding to $\alpha \in \mathfrak{K}$) which fixes $v$ and
\[ \mu_\alpha(v) = -(\rho,\lambda) \text{  and  } ||\mu(v)|| =\frac{\mu_\alpha(v)}{||\alpha||} = \left(\mu^*(v),\frac{\alpha}{||\alpha||}\right). \]
Hence $\beta := \mu^*(v)$ is a positive scalar multiple of $\alpha$ and $\alpha_v =0$ as $\lambda$ fixes $v$. Therefore, $\beta_v = 0$ and so $v$ is critical for $\mu_\beta$ where $\beta = \mu^*(v)$.
\end{proof}

\begin{cor}\label{cor F}
 Let $v \in V$ and $\beta = \mu^*(v) \neq 0$. If $\beta_v = 0$, then $v$ is unstable. Moreover, if $\lambda$ is a 1-PS associated to $\beta$, then $\lambda$ is adapted to $v$.
\end{cor}
\begin{proof}
Let $\lambda : S^1 \ra K$ be the 1-PS associated to $c\beta$ where $c$ is a positive number such that $c\beta$ is integral; we note that this is non-trivial as $\beta \neq 0$. The 1-PS $\lambda$ fixes $v$ and so by Lemma \ref{lemma E} i) we have that
\[ 0 < c ||\beta||^2 = \mu_{c\beta}(v) = -(\rho,\lambda) = -\mu(v,\lambda).\]
Hence $v$ is unstable by Proposition \ref{HM crit}. Moreover,
\[ M(v) \leq \frac{\mu(v,\lambda)}{||\lambda||} = -\frac{\mu_{c\beta}(v)}{||c\beta||} = - ||\beta|| = -||\mu(v)|| \]
and so by Lemma \ref{lemma E} ii), $M(v) = -||\mu(v)||$ and $\lambda$ is adapted to $v$.
\end{proof}

We now have a nice description of the non-minimal critical points of $||\mu||^2$ (for the projective version, see \cite{ness} Theorem 6.1):

\begin{thm}\label{thm G}
Let $v \in V$ and $\beta = \mu^*(v)$. Then the following are equivalent: 
\begin{enumerate} \renewcommand{\labelenumi}{\roman{enumi})}
 \item $v$ is a non-minimal critical point of $||\mu||^2$.
 \item $\beta \neq 0$ and $\beta_v = 0$.
 \item $-M(v) = ||\beta|| >0$
\end{enumerate}
\end{thm}
\begin{proof}
The equivalence of i) and ii) is given in Lemma \ref{crit} and the equivalence between ii) and iii) is given by Lemma \ref{lemma E} ii).
\end{proof}

Let $\beta \neq 0$ be an element of $\mathfrak{K}$, such that the adjoint orbit $K \cdot \beta \in \mathcal{B}$ indexes a Morse stratum $S_{K \cdot \beta}$. As $\beta$ is a rational weight, we let $n$ be the smallest positive integer such that $n \beta$ is integral and we let $\lambda:=\lambda_\beta$ be the indivisible 1-PS (considered as a 1-PS $\lambda : S^1 \ra K$ or $\lambda : \CC^* \ra G$) associated to this integral weight. 

In $\S$\ref{sec mubeta}, we considered the action of the stabiliser subgroup $K_\beta$ of $\beta$ under the adjoint action of $K$ on the critical locus $V^{\beta} := \text{crit} \mu_\beta$, for which there is a moment map $\mu - \beta^* : V^\beta \ra \mathfrak{K}_\beta^*$. We defined $Z_\beta^{\min} \subset V^\beta$ to be the minimal Morse stratum for $||\mu - \beta^*||^2$ whose corresponding critical subset is $C_\beta := V^\beta \cap \mu^{-1}(\beta^*)$ by Lemma \ref{lemma H}. 

In $\S$\ref{sec hess}, we considered the action of the subgroup $G_\lambda$ consisting of elements of $G$ which commute with $\lambda$ on the fixed point locus $V^{\lambda}$ and used a character $\rho_{\lambda}$ of $G_\lambda$ defined at (\ref{defn rholambda}) to linearise this action. The semistable set for this action (with respect to $\rho_{\lambda}$) is equal to the limit set $Z_\lambda$ defined at (\ref{Zlambda}) of the blade $S_\lambda$ of the Hesselink stratum $S_{[\lambda]}$ associated to this 1-PS $\lambda$ by Proposition \ref{prop on Zlambda}. 

For $\lambda = \lambda_\beta$, we note that $V^{\lambda}=V^{\beta}$ and $G_\lambda$ is the complexification of $K_\beta$; thus from now on we shall write $G_\beta := G_\lambda$ for $\lambda = \lambda_\beta$.

\begin{thm}\label{thm H}
Let $\beta \neq 0$ be a point in $\mathfrak{K}$ which indexes a nonempty Morse stratum and let $\lambda = \lambda_\beta$ be the associated 1-PS as above. Then $Z_\lambda = Z_\beta^{\min}$ and the set of closed $G_\beta$-orbits in $Z_\lambda$ is equal to $G_{\beta} C_\beta$.
\end{thm}
\begin{proof}
We recall that $\lambda$ is the 1-PS associated to the integral weight $n\beta$ for some $n >0$ and so $n ||\beta|| = ||\lambda||$. If $v \in C_\beta$, then by Corollary \ref{cor F} the 1-PS $\lambda$ is adapted to $v$ and by Theorem \ref{thm G} $-M(v) = ||\beta||$. Hence,
\begin{equation}\label{eqn Z}
 -M(v) = -\frac{(\rho,\lambda)}{||\lambda||} = ||\beta|| = \frac{||\lambda||}{n}.
\end{equation}
As moment maps for the action of $K_\beta$ on $V^{\beta}$ we claim that $\mu - \beta^* $ is equal to the natural moment map for this action shifted by $\rho_\lambda/||\lambda||^2$ which we denote by $\mu'$. For $v \in V^{\beta}$ and $\alpha \in \mathfrak{K}_\beta$,
\[
 (\mu - \beta^*)(v)\cdot \alpha := \mu(v) \cdot \alpha - \beta^* \cdot \alpha = \frac{1}{2 \pi i}\left(  H(\alpha v, v) - d \rho \cdot \alpha \right) - \beta^* \cdot \alpha 
\]
 and
\[ \mu'(v)\cdot \alpha := \frac{1}{2\pi i} \left(H(\alpha v,v) - \frac{1}{||\lambda||^2} d\rho_\lambda \cdot \alpha \right)= \frac{1}{2\pi i} \left(H(\alpha v,v) - d {\rho}\cdot \alpha  + \frac{(\rho,\lambda)}{||\lambda||^2} d\lambda^* \cdot \alpha \right)\]
where $\lambda^* : K \ra S^1$ is the character dual to $\lambda : S^{1} \ra K$ under our fixed inner product (we note that as $\lambda$ is the 1-PS associated to the rational weight $n\beta$ we have that $d \lambda^* = 2n \pi i \beta^*$).
It follows from the relation given at (\ref{eqn Z}), that
\[  \frac{1}{2 \pi i} \frac{(\rho,\lambda) }{||\lambda||^2} d \lambda^* = -\frac{1}{2 n \pi i } d \lambda^* = - \beta^* \] 
which proves the claim and so $\mu' = \mu - \beta^*$ as moment maps for the $K_\beta$-action on $V^\beta$.

For any character $\rho$ and $m>0$ we have that the GIT (semi)stable sets for $\rho$ and $m\rho$ agree: \[V^{\rho\text{-(s)s}} = V^{m\rho\text{-(s)s}}\]
(the easiest way to see this is to use the description of (semi)stability coming from the Hilbert--Mumford criterion given in Proposition \ref{HM crit}). We note that although $\rho_\lambda$ is an honest character of $G_\lambda$, the element $\rho_\lambda/||\lambda||^2$ is only a rational character. However for the purposes of GIT, we can still use this rational character to linearise the action and from the observation above we note that
\begin{equation}\label{GIT ssets} (V^\lambda)^{\rho_\lambda/||\lambda||^2\text{-ss}} = (V^\lambda)^{\rho_\lambda\text{-ss}} = Z_\lambda.\end{equation} 
where the second equality comes from Proposition \ref{prop on Zlambda}. Then by part iv) of the affine Kempf--Ness theorem applied to the moment map $\mu'$ for the action of $K_\beta$ on $V^{\beta}$, we have that the GIT semistable set of (\ref{GIT ssets}) is the lowest Morse stratum $Z_\beta^{\min}$ for the norm square of the moment map $\mu'=\mu - \beta^*  : V_\beta \ra \mathfrak{K}_\beta^*$. The final statement is part i) of the affine Kempf--Ness theorem.
\end{proof}

\begin{cor}\label{cor I}
Let $\beta \neq 0$ be a point in $\mathfrak{K}$ which indexes a nonempty Morse stratum and let $\lambda = \lambda_\beta$ be the associated 1-PS as above; then \[S_\lambda = Y_\beta^{\min}.\]
\end{cor}
\begin{proof}
 This follows immediately from the result above given the fact that $p_\lambda : V^\lambda_+ \ra V^\lambda$ is equal to $p_\beta : V^\beta_+ \ra V^\beta$ and $S_\lambda = p_\lambda^{-1}(Z_\lambda)$ and $Y_\beta^{\min}=p_\beta^{-1}(Z_\beta^{\min})$.
\end{proof}

We have seen how to associate to $\beta \in \mathfrak{K}$ (which indexes a Morse stratum $S_{K\cdot \beta}$) a 1-PS $\lambda_\beta$ which indexes a Hesselink stratum $S_{[\lambda_\beta]}$ and Theorem \ref{thm H} shows that the associated Hesselink stratum is nonempty. We now describe how to associate to a 1-PS $\lambda$ indexing an unstable Hesselink stratum $S_{[\lambda]}$ a rational weight $\beta=\beta(\lambda)$ which indexes a Morse stratum. Firstly, we can assume (by conjugating $\lambda$ if necessary) that $\lambda(S^1) \subset K$ and from above we know that $\beta$ should be a positive scalar multiple of $d\lambda (2 \pi i) \in \mathfrak{K}$. If we write $\beta = c d \lambda(2 \pi i)$ we have
\begin{equation}\label{eqn Y} c^2 ||\lambda||^2 =||\beta||^2 = - d \rho \cdot \beta / 2 - \frac{1}{2 \pi i} d \rho \cdot \beta = -c(\rho,\lambda) \end{equation}
by Remark \ref{norm of beta} and so
\[ \beta := -\frac{(\rho,\lambda)}{||\lambda||^2} d\lambda (2 \pi i) \in \mathfrak{K} \]
is the desired element which satisfies (\ref{eqn Y}). Moreover, it follows that the associated Morse stratum $ S_{K \cdot \beta}$ is nonempty by Theorem \ref{thm H} above.

\begin{thm}\label{Hess is Morse}
 Let $G = K_\CC$ be a complex reductive group acting linearly on a complex vector space $V$ for which $K$ acts unitarily with respect to a fixed Hermitian inner product on $V$. We use a character $\rho$ to linearise the action and also to shift the natural moment map associated to this action. For a fixed $K$-invariant inner product on the Lie algebra $\mathfrak{K}$, the Morse stratification and Hesselink's stratification coincide; that is,
\[ S_{K \cdot \beta} = S_{[\lambda_\beta]}\]
where $\lambda_\beta$ is the 1-PS associated to $\beta$ (or $\beta$ is the rational weight associated to the 1-PS $\lambda_\beta$).
\end{thm}
\begin{proof}
We have seen above how to associate to $\beta$ a 1-PS $\lambda_\beta$ and conversely how to associate to a 1-PS $\lambda$ a rational weight $\beta(\lambda)$. The theorem follows from Corollary \ref{cor I} above and the fact that $S_{[\lambda_\beta]} = G S_{\lambda_\beta}$ and $S_{K \cdot \beta} = GY_{\beta}^{\min}$ (which we prove in Theorem \ref{thm alt desc} below).
\end{proof}

\subsection{The proof of the alternate description of the Morse strata}

In this section we prove the alternative description of the Morse strata; that is, $S_{K \cdot \beta}= GY_\beta^{\min}$ where we continue to use the notation defined in $\S$\ref{sec mubeta}. We fix a positive Weyl chamber $\mathfrak{t}_+$ and let $\beta$ denote the unique point of the adjoint orbit $ K \cdot \beta $ which meets $\mathfrak{t}_+$.

\begin{prop}\label{prop K}
For $v \in GV_+^{\beta}$, we have $||\mu(v)|| \geq ||\beta||$. Moreover, if $v \in GY_\beta^{\min}$ then $\beta$ is the closest point to the origin of $\mu^*(\overline{G \cdot v}) \cap \mathfrak{t}_+$.
\end{prop}
\begin{proof}
 For the first statement, we can assume $v \in V_+^{\beta}$ as $||\mu||^2$ is $K$-invariant and $GV_+^{\beta} = KV_+^{\beta}$. It follows from Lemma \ref{lemma D} that \[\mu(v) \cdot \beta \geq - \frac{1}{2\pi i}d \rho \cdot \beta \]
where the right hand side is equal to $||\beta||^2$ by Remark \ref{norm of beta}. Therefore, $||\mu(v)|| \geq ||\beta||$. For the second statement, we can assume $v \in Y_\beta^{\min}$ and as $||\mu(v)|| \geq ||\beta||$ it suffices to show that $\beta \in \mu^*(\overline{G \cdot v})$. Since $p_\beta(v) \in Z_\beta^{\min}$ is contained in the orbit closure $\overline{G \cdot v}$, it suffices to show that $\beta \in \mu^*(\overline{G \cdot v})$ for $v \in Z_\beta^{\min}$. As $Z_\beta^{\min}$ is the minimal Morse stratum for $||\mu - \beta^* ||^2$ with critical subset $C_\beta$ and the flow is also contained in $G$-orbit closures, it suffices to show that $\beta \in \mu^*(\overline{G \cdot v})$ for points $v \in C_\beta$. Since $\mu^*$ takes the value $\beta$ on $C_\beta$, this completes the proof.
\end{proof}

The above proposition is an affine version of \cite{kirwan} Corollary 6.11 and Corollary 6.12. For the proof of Theorem \ref{thm alt desc} below, we take a slightly different approach to \cite{kirwan}.

\begin{thm}\label{thm alt desc}
For $\beta \in \mathfrak{K}$, we have $S_{K \cdot \beta} = GY_\beta^{\min}$.
\end{thm}
\begin{proof}
We shall prove that $GY_\beta^{\min} \subset S_{K \cdot \beta}$ as then these sets must be equal as we can write $V$ as a disjoint union of the Morse strata $S_{K \cdot \beta}$ and also as a disjoint union of the subsets $GY_\beta^{\min}$ by Proposition \ref{prop K}. Let $v \in GY_\beta^{\min}$ and suppose that $v$ belongs to some other Morse stratum indexed by ${\beta'} \in \mathfrak{t}_+$. As the Morse strata are $K$-invariant, we can assume without loss of generality that the negative gradient flow of $v$ under $||\mu||^2$ converges to $C_{\beta'}$ and so there is a point $v'$ in the orbit closure of $v$ such that $\mu^*(v') = \beta'$. Then by Proposition \ref{prop K}, we have that 
\begin{equation}\label{eqn a}
 || \beta'|| > ||\beta||.
\end{equation}
As $v \in GY_\beta^{\min}$, we can write $v = gy$ for $y \in Y_\beta^{\min}$. Then $p_\beta(y) \in Z_\beta^{\min}$ and $p_\beta(v)$ flows under the negative gradient flow of $||\mu - \beta^* ||^2$ to a point $c \in C_\beta$. Then $c \in \overline{G \cdot v}\cap S_{K \cdot \beta}$ and so $\partial S_{K \cdot \beta'} \cap S_{K \cdot \beta} \neq \emptyset$, but as the Morse strata form a stratification this contradicts (\ref{eqn a}). Therefore $GY_\beta^{\min} \subset S_{K \cdot \beta}$ as required. 
\end{proof}

\section{Stratifications of spaces of quiver representations}\label{sec quiv}

As mentioned above, there is a construction due to King of the moduli space of \lq semistable' representations of a quiver $Q$ with fixed invariants $d$ as a GIT quotient of a reductive group $G$ acting on an affine space $\text{Rep} (Q,d)$ with respect to a character $\rho$ \cite{king}. The notion of (semi)stability is determined by a stability parameter $\theta$ which is also used to construct the character $\rho $. 

Over the complex numbers, we know that the Morse stratification agrees with Hesselink's stratification for the action of $G = K_\CC$ on the complex vector space $\text{Rep} (Q,d)$ linearised by $\rho$ for any invariant inner product on $\mathfrak{K}$. In this section we describe another stratification of $\text{Rep} (Q,d)$ by Harder--Narasimhan types $\tau$ where the Harder--Narasimhan filtration of a quiver representation is defined by using both the stability parameter $\theta$ and the chosen inner product on $\mathfrak{K}$ (cf. Definition \ref{defn HN} below). We shall see that a Harder--Narasimhan stratum $S_\tau$ has a description as $S_{\tau} = G Y_\tau^{\text{ss}}$ and there is a retraction $p_\tau : Y_\tau^{\text{ss}} \ra Z_\tau^{\text{ss}}$ where $Z_\tau^{\text{ss}}$ can be explicitly described. Moreover, we shall see that the Harder--Narasimhan stratification, Morse stratification and Hesselink stratification all agree.

We briefly summarise here the known results on Harder--Narasimhan stratifications for quivers. Reineke describes a Harder--Narasimhan stratification for quivers and obtains formulae for the Betti numbers of the associated moduli spaces \cite{reineke}. However, Reineke's notion of Harder--Narasimhan filtration does not depend on a choice of invariant inner product on $\mathfrak{K}$ (the definition given by Reineke corresponds to the Harder--Narasimhan filtration associated to the Killing form in our definition below). Harada and Wilkin \cite{haradawilkin} show for the Killing form that the Morse stratification agrees with the Harder--Narasimhan stratification. The group $G$ is a product of general linear groups and so one can construct families of invariant inner products on $\mathfrak{K}$ by using a weighted sum of the Killing forms on the unitary Lie algebras which make up $\mathfrak{K}$. This idea is used by Tur \cite{tur} in his thesis, where he shows that the Harder--Narasimhan stratification and Hesselink's stratification coincide for any inner product associated to a collection of weights. It follows from this result and Theorem \ref{Hess is Morse} above that all three stratifications coincide, however we provide a short proof of this result below for completeness.

\subsection{GIT construction of quiver moduli}

A ($\CC$-)representation of a quiver $Q=(V,A,h,t)$ is a tuple $W = (W_v, \phi_e)$ consisting of a complex vector space $W_v$ for each vertex $v$ and a linear map $\phi_a : W_{t(a)} \ra W_{h(a)}$ for each arrow $a$. A morphism of quiver representations $f: W \ra W'$ is given by linear maps $f_v : W_v \ra W_v'$ for each vertex $v$ such that $f_{h(a)} \circ \phi_a = \phi'_a \circ f_{t(a)}$ for every arrow $a$. One forms the obvious notions of isomorphism, subrepresentation and quotient representation. The dimension of a quiver representation is $ \dim W = (\dim W_v) \in \NN^V.$

For a dimension vector $d = (d_v) \in \NN^V$, the complex affine space
\[ \text{Rep}(Q,d) = \oplus_{a \in A} \text{Hom}(\CC^{d_{t(a)}},\CC^{d_{h(a)}}) \]
parametrises representations $W$ of $Q$ of dimension $d$ with a choice of isomorphism $W_v \cong \CC^{d_v}$. The group
\[ G(Q,d) = \prod_{v \in V} \GL(d_v)\]
acts naturally on $\text{Rep}(Q,d)$ by conjugation: $(g \cdot \varphi)_a = g_{h(a)} \varphi_a g_{t(a)}^{-1}$ for $g=(g_v) \in G(Q,d)$ and $\varphi = (\varphi_a) \in \text{Rep}(Q,d)$. We note that there is a copy of the multiplicative group $\CC^*$ embedded in $G(Q,d)$ as $t \mapsto (tI_{d_v})$ which acts trivially on $\text{Rep}(Q,d)$ and so if we want to have stable points we must consider the action of $G(Q,d)/\CC^*$ or modify the definition of stability so that stable orbits can have positive dimensional stabilisers (the second approach is taken by King in \cite{king}). However, for the purposes of studying stratifications associated to this action, the presence of this copy of $\CC^*$ does not matter and so we work with the action of $G=G(Q,d)$.

The action is linearised by choosing a character $\rho$ of $G$. Following King \cite{king}, we let $\theta = (\theta_v) \in \ZZ^V$ denote a tuple of integers such that $\sum_v \theta_v d_v = 0$ and associate to $\theta$, the character $\rho = \rho_\theta :G \ra \CC^*$ given by
\[ \rho (g_v) = \prod_v \det(g_v)^{\theta_v}.\]
King uses the Hilbert--Mumford criterion to reinterpret the notion of $\rho$-semistability for points in $\text{Rep}(Q,d)$ as a condition for the corresponding quiver representation:

\begin{defn}\label{theta ss defn}
A representation $W$ of $Q$ of dimension $d$ is $\theta$-semistable if for all proper subrepresentations $W' \subset W$ we have $\theta(W'):= \sum_v \theta_v \dim W_v' \geq 0$. 
\end{defn}

The moduli space of $\theta$-semistable quiver representations is constructed by King as the quotient of $G$ acting on $\text{Rep}(Q,d) $ with respect to the character $\rho$ defined by $\theta$.

\subsection{Harder--Narasimhan filtrations for quivers}

As $G =K_\CC= G(Q,d)$ is a product of general linear groups $\GL(d_v)$, the Lie algebra $\mathfrak{K}$ is a sum of unitary Lie algebras $\mathfrak{u}(\CC^{d_v})$. As every invariant inner product on $\mathfrak{u}(\CC^{d_v})$ is a positive scalar multiple of the Killing form $\kappa_v$, the invariant inner products $(-,-)$ on $\mathfrak{K}$ are weighted sums $\sum \alpha_v \kappa_v$ for positive $\alpha_v$. Moreover, we shall assume $\alpha_v$ are integral so that the norm squared of an integral element in $\mathfrak{K}$ is integral (cf. Assumption \ref{int ass}). For $\alpha = (\alpha_v) \in \NN^V_+$, we let $(-,-)_\alpha$ denote the associated inner product.

\begin{defn}\label{defn HN}
For a representation $W$ of $Q$ (of any dimension), we say $W$ is $\theta$-semistable if for all proper subrepresentations
\[ \frac{\theta(W')}{\alpha(W')} \geq \frac{\theta(W)}{\alpha(W)} \]
where $\alpha(W) := \sum \alpha_v \dim W_v$. A Harder--Narasimhan filtration of $W$ (with respect to $\alpha$ and $\theta$) is a filtration $0 = W_{(0)} \subset W_{(1)} \subset \cdots \subset W_{(s)} =W$ by subrepresentations, such that the quotient represenations $W_i := W_{(i)}/W_{(i-1)}$ are $\theta$-semistable and
\[ \frac{\theta(W_1)}{\alpha(W_1)} < \frac{\theta(W_2)}{\alpha(W_2)} <\cdots < \frac{\theta(W_s)}{\alpha(W_s)} .\]
The Harder--Narasimhan type of $W$ (with respect to $\alpha$ and $\theta$) is $\tau(W) := (\dim W_1, \dots , \dim W_s)$.
\end{defn}

The standard techniques are used to show existence and uniqueness of the Harder--Narasimhan filtration. The trivial Harder--Narasimhan type for representations of dimension $d$ is $\tau_0 = (d)$ and the representations with this type are $\theta$-semistable. We have a decomposition
\[ \text{Rep}(Q,d) = \bigsqcup_{\tau} S_{\tau} \]
where $S_{\tau}$ indexes the subset of representations with Harder--Narasimhan type $\tau$.

\subsection{All three stratifications coincide}\label{sec 3 strat}

Let $\tau = (d_1, \dots ,d_s)$ be a Harder--Narasimhan type (with respect to $\alpha$ and $\theta$) for a representation of $Q$ of dimension $d$. We write $W_v:=\CC^{d_v} \cong W_{1,v} \oplus \cdots \oplus W_{s,v}$ where $W_{i,v}=\CC^{(d_i)_v} $, then every $(\varphi_a) \in \text{Rep}(Q,d)$ can be written as 
\[ \varphi_a = \left(\begin{array}{ccc}\varphi_a^{11} & \dots & \varphi_a^{1s} \\ \vdots & \ddots & \vdots \\ \varphi_a^{s1} & \dots & \varphi_a^{ss}\end{array}\right) \]
where $\varphi_a^{ij} :  W_{j,t(a)} \ra W_{i,h(a)}$. We let
\[ Z_\tau := \{ (\varphi_a) \in \text{Rep}(Q,d) : \varphi_a^{ij} = 0 \text{ for } i \neq j \} \cong \bigoplus_{i=1}^s \text{Rep}(Q,d_i) \]
and \[ Y_\tau := \{ (\varphi_a) \in \text{Rep}(Q,d) : \varphi_a^{ij} = 0 \text{ for } i > j \} ;\] thus we have a projection $p_\tau : Y_\tau \ra Z_\tau$. Let $\text{Rep}(Q,d_i)^{\theta\text{-ss}}\subset \text{Rep}(Q,d_i)$ denote the open subset of $\theta$-semistable representations of dimension $d_i$. Then we define $Z_\tau^{\text{ss}} \subset Z_\tau$ to be the image of $\oplus_{i} \text{Rep}(Q,d_i)^{\theta\text{-ss}}$ under the above isomorphism and $Y_\tau^{\text{ss}} :=p_\tau^{-1}(Z_\tau^{\text{ss}})$. By construction of $Y_\tau^{\text{ss}}$, every point in $Y_\tau^{\text{ss}}$ has a Harder--Narasimhan filtration of type $\tau$.

\begin{defn} We define $\beta =\beta(\tau) \in \mathfrak{K} $ by 
\[ \beta_v = 2 \pi i \: \text{diag}(\beta_1, \dots , \beta_1,\beta_2, \dots , \beta_2, \dots ,\beta_s, \dots , \beta_s) \in \mathfrak{u}(\CC^{d_v})
\]
where $\beta_i := - \theta(d_i)/\alpha(d_i)$ appears $(d_{i})_v$ times in $\beta_v$.\end{defn}

We let $\lambda_\beta$ denote the 1-PS associated to $\beta$ so $d \lambda_\beta(2 \pi i) = n\beta$ for some $n >0$ and we let $S_{[\lambda_\beta]}$ denote the Hesselink strata for $(-,-)_\alpha$. 

\begin{prop}\label{lemma P}
Let $\beta = \beta(\tau)$; then 
\begin{enumerate}
\renewcommand{\labelenumi}{\roman{enumi})}
 \item $V^{\lambda_\beta} = Z_\tau$.
\item $V_+^{\lambda_\beta} = Y_\tau$.
\item $Z_{\lambda_\beta} = Z_\tau^{\text{ss}}$.
\item $S_{\lambda_\beta} = Y_\tau^{\text{ss}}$.
\end{enumerate}
\end{prop}
\begin{proof}
The first two statements follow immediately from the definition of $\beta$ where we note that $\beta_1 > \dots >
 \beta_s$ as $\tau$ is a Harder--Narasimhan type and so the \lq slopes' $\theta/\alpha$ are increasing. As iv) follows from iii), it suffices to prove iii). Let $\varphi \in Z_\tau^{\text{ss}}$; then by Proposition \ref{prop on Zlambda} we know that $\varphi \in Z_{\lambda_\beta}$ if for every 1-PS $\lambda'=(\lambda'_{i,v})$ of $G_{\lambda_\beta} = \Pi_{i,v} \GL(W_{i,v})$ for which $\lim_{t \to 0} \lambda'(t) \cdot \varphi$ exists we have 
\begin{equation}\label{eqn d} \mu^{\rho_{\lambda_\beta}}(\varphi,\lambda') =||\lambda_\beta||^2_\alpha(\rho,\lambda') - (\rho,\lambda_\beta)(\lambda_\beta,\lambda')_\alpha =||\lambda_\beta||^2_\alpha \left[ (\rho,\lambda') +(\beta,d\lambda'(2 \pi i))_\alpha  \right] \geq 0\end{equation}
where the second equality follows as $\beta$ is scaled so that $||\beta||^2_\alpha=-\frac{1}{2 \pi i}d\rho \cdot \beta =-\frac{1}{n}(\rho,\lambda_\beta)$. 

We simultaneously diagonalise the action of each $\lambda'_{i,v}$ on $W_{i,v}$, so we have weights $\gamma_1 > \dots > \gamma_r$ and decompositions $W_{i,v} = W^{1}_{i,v} \oplus \cdots \oplus W^{r}_{i,v}$ such that $\lambda_{i,v}(t)$ acts on $W^{j}_{i,v}$ by $t^{\gamma_j}$. We note that $\lim_{t \to 0} \lambda'(t) \cdot \varphi$ exists if and only if $W_{i}^{(j)}=(W^{1}_{i,v} \oplus \cdots \oplus W^{j}_{i,v}, \varphi_a^{ii}|)$ is a subrepresentation of $W_i=(\CC^{(d_i)_v}.\varphi_a^{ii})$ for all $i$ and $j$. Then (\ref{eqn d}) is equivalent to
\begin{equation}\label{eqn e} \sum_{i,j} \gamma_j \left(\theta(W_i^j) +\beta_i \alpha(W_i^j)\right) = \sum_{i,j} \gamma_j \left( \theta(W_i^j) - \frac{\theta(d_i)}{\alpha(d_i)} \alpha(W_i^j) \right) \geq 0. \end{equation}                                                                                                                                                                                                 As $\varphi \in Z_\tau^{\text{ss}}$, the representations $W_i=(W_{i,v}.\varphi_a^{ii})$ of dimension $d_i$ are $\theta$-semistable. Moreover, as $W_i^{(j)}$ are subrepresentations of $W_i$, semistability for all $i$ implies that (\ref{eqn e}) holds. 

Conversely if $\varphi \in Z_\tau - Z_\tau^{\text{ss}}$, then there must be a destabilising subrepresentation $W_i^1$ of $W_i$ for some $1 \leq i \leq s$. We can pick an orthogonal complement $W_{i,v}^{2}$ to $W_{i,v}^1 \subset W_{i,v}$ and define a 1-PS $\lambda'$ of $G_{\lambda_\beta}$ by
\[ \lambda_{i,v}(t) = \left( \begin{array}{cc} t I_{W_{i,v}^1} & 0 \\ 0 & I_{W_{i,v}^2} \end{array} \right) \text{  and  } \lambda_{k,v}(t) = I_{W_{k,v}} \text{ for } k \neq i.\]
Then as $W_i^1$ destabilises $W_i$, we have
\[ \frac{\mu^{\rho_{\lambda_\beta}}(\varphi,\lambda')}{||\lambda_\beta||^2_\alpha} = (\rho,\lambda') +(\beta,d\lambda'(2 \pi i))_\alpha = \theta(W_i^1) + \beta_i \alpha(W_i^1) <0\]
which proves $\varphi \notin Z_{\lambda_\beta}$ by Proposition \ref{prop on Zlambda}.
\end{proof}

\begin{thm}
 The Harder--Narasimhan stratification of the space of representations of a quiver of fixed dimension agrees with both the Morse and Hesselink stratification for $(-,-)_\alpha$. Moreover, the Harder--Narasimhan strata have the form $S_\tau = G Y_\tau^{\text{ss}}$. 
\end{thm}
\begin{proof}
 Let $\tau$ be a Harder--Narasimhan type as above; then clearly $Y_\tau^{\text{ss}} \subset S_\tau$ and as the Harder--Narasimhan strata are $G$-invariant we have $G Y_\tau^{\text{ss}} \subset S_\tau$. However, by Proposition \ref{lemma P} above $G Y_\tau^{\text{ss}} = GS_{\lambda_\beta}=S_{[\lambda_\beta]}$ where $\beta = \beta(\tau)$. Since the Harder--Narasimhan strata and the Hesselink strata both form a stratification of $\text{Rep}(Q,d)$ and every Hesselink stratum is contained in a Harder--Narasimhan stratum, these stratifications must coincide and $S_\tau = G Y_\tau^{\text{ss}}$. It follows from Theorem \ref{Hess is Morse} that all three stratifications coincide.
\end{proof}

\bibliographystyle{amsplain}
\bibliography{references}

\medskip \medskip

\noindent \textsc{Institut f\"{u}r Mathematik, Universit\"{a}t Z\"{u}rich,}\\
\textsc{Winterthurerstrasse 190, CH-8057, Z\"{u}rich, Switzerland}
\medskip

\noindent \texttt{victoria.hoskins@math.uzh.ch}   

\end{document}